\documentclass[11pt]{amsart}
\usepackage{amssymb}
\usepackage{amsmath}
\usepackage{mathrsfs}
\usepackage{faktor}
\usepackage{tikz}

\newcommand{\upi}{\underline{\pi}}

\usetikzlibrary{arrows}
\usetikzlibrary{shapes.geometric}
\usetikzlibrary{decorations.pathreplacing,decorations.markings}
\usepackage[utf8]{inputenc}
\usepackage[top=1in,bottom=1in,right=1in,left=1in]{geometry}
\usepackage{hyperref}

\newtheorem{theorem}{Theorem}[section]
\newtheorem{lem}[theorem]{Lemma}
\newtheorem{proposition}[theorem]{Proposition}
\newtheorem{cor}[theorem]{Corollary}

\newtheorem{conjecture}[theorem]{Conjecture}

\theoremstyle{definition}
\newtheorem{dfn}[theorem]{Definition}
\newtheorem{ex}[theorem]{Example}
\newtheorem{rmk}[theorem]{Remark}
\newtheorem{ntn}[theorem]{Notation}

\numberwithin{theorem}{section}

\newenvironment{theorem_no_number}[1][]{\begin{trivlist}
\item[\hskip \labelsep {\bfseries Theorem \def\temp{#1}\ifx\temp\empty  #1\else  #1\fi
.}] \itshape}  {\end{trivlist}}
\newenvironment{prop_no_number}[1][]{\begin{trivlist}
\item[\hskip \labelsep {\bfseries Proposition \def\temp{#1}\ifx\temp\empty  #1\else  #1\fi
.}] \itshape}  {\end{trivlist}}
\newenvironment{cor_no_number}[1][]{\begin{trivlist}
\item[\hskip \labelsep {\bfseries Corollary \def\temp{#1}\ifx\temp\empty  #1\else  #1\fi			.}] \itshape}  {\end{trivlist}}

\renewenvironment{proof}[1][]{\begin{trivlist}
\item[\hskip \labelsep {\bfseries Proof  \def\temp{#1}\ifx\temp\empty  #1\else  (#1)\fi
}]}{\hfill\(\square\) \end{trivlist}}

\DeclareMathOperator{\Aut}{Aut}
\DeclareMathOperator{\End}{End}

\DeclareMathOperator{\GL}{GL}

\DeclareMathOperator{\modns}{mod}

\newcommand{\Z}{\mathbb{Z}}
\newcommand{\Q}{\mathbb{Q}}
\newcommand{\C}{\mathbb{C}}
\newcommand{\R}{\mathbb{R}}

\newcommand{\cA}{\mathcal{A}}

\newcommand{\cP}{\mathcal{P}}
\newcommand{\cQ}{\mathcal{Q}}
\newcommand{\cR}{\mathcal{R}}
\newcommand{\cS}{\mathcal{S}}
\newcommand{\wt}[1]{\widetilde{#1}}
\newcommand{\mbS}{\mathbb{S}}

\title{Equations in virtually abelian groups: languages and growth}

\author{Alex Evetts and Alex Levine}

\address{Department of Mathematics, The University of Manchester, Manchester, M13 9PL, UK}

\email{alex.evetts@manchester.ac.uk}

\address{School of Mathematical and Computer Sciences, Heriot-Watt University,
Edinburgh, EH14 4AS, UK}

\email{aal7@hw.ac.uk}



\begin{document}

\begin{abstract}
	This paper explores the nature of the solution sets of systems of equations in
	virtually abelian groups. We view this question from two angles. From a formal
	language perspective, we prove that the set of solutions to a system of
	equations forms an EDT0L language, with respect to a natural normal form.
	Looking at growth, we show that the growth series of the language of solutions
	is rational. Furthermore, considering the set of solutions as a set of tuples
	of group elements, we show that it has rational relative growth series with
	respect to any finite generating set.
\end{abstract}

\maketitle

\section{Introduction}
	An \textit{equation} with set of variables
	\(V\) in a group \(G\) is an identity \(w = 1\), for any \(w \in G \ast F_V\),
	where \(F_V\) denotes the free group on \(V\).
	Equations in groups have been a significant area of research, particularly
	since Makanin proved that it is decidable whether a system of equations in a
	free group admits a solution, in a series of papers including
	\cite{Makanin_systems}, \cite{Makanin_semigroups} and
	\cite{Makanin_eqns_free_group}. Following on from this,
	Razborov described the set of solutions to systems of equations in free groups
	(\cite{Razborov_thesis}, \cite{Razborov_english}). More recently Ciobanu,
	Diekert and Elder \cite{eqns_free_grps} proved that the solution set to a
	system of equations in a free group is an EDT0L language (described as reduced
	words over a basis). This was generalised to virtually free groups in
	\cite{VF_eqns}, and to all hyperbolic groups in \cite{eqns_hyp_grps}. Diekert,
	Je\.{z} and Kufleitner \cite{EDT0L_RAAGs} showed that this also holds for
	right-angled Artin groups.

%	Equations in a group \(G\) can be used to understand epimorphisms from a
%	finitely presented group \(H = \langle \Sigma \mid R \rangle\) onto \(G\).
%	This is done by taking each relation \(w \in R\), and considering it to be an
%	equation in \(G\), with letters in \(\Sigma\) as the set of variables. This
%	system admits a solution, if and only if there is a surjective function \(\phi
%	\colon H \to G\) that satisfies the relations of \(H\).

	The class of EDT0L languages was introduced by Rozenberg \cite{ET0LDef} in 1973. They are an example of L-systems, which where themselves introduced in
	order to model growth of organisms. A key fact about this growth is that it
	occurs in parallel across the organism, and this is reflected in the
	definition of EDT0L systems. Every regular language is EDT0L, and every EDT0L
	language is an indexed language, and thus context-sensitive. There exist
	context-free languages which are not EDT0L, and EDT0L languages which are not
	context-free.

	The use of EDT0L languages in describing solutions to equations, rather than
	the more studied class of context-free languages, arises from the fact that
	equations can have arbitrarily many variables. When expressing the solution
	\((x_1, \ \ldots, \ x_n)\) as a word, we write it in the form \(x_1 \# \cdots
	\# x_n\), for some new letter \(\#\). It follows that even in
	\(\mathbb{Z}=\langle a \mid \rangle\), equations with \(3\) or more variables
	will not necessarily be context-free, as the system \(X = Y = Z\) will have
	the solution language \(\{a^m \# a^m \# a^m \mid m \in \mathbb{Z}\}\), which
	is not a context-free language over \(\{a, \ a^{-1}, \ \#\}\).

	It has long been regarded as `folklore' that it is decidable whether systems
	of equations in virtually abelian groups admit solutions, however it is
	unclear when this was first proved. In \cite{elementary_gp_theories} the
	stronger result that virtually abelian groups have decidable first order
	theory is shown. A more direct proof of the solubility of equations in
	virtually abelian groups can be found in Lemma 5.4 of \cite{Dahmani}. In this
	paper we study the properties of solution sets of systems of equations in finitely generated
	virtually abelian groups. Such sets are also known as \emph{algebraic sets}.

	Given a choice of finite generating set, and a corresponding normal form, we
	study the language of representatives for algebraic sets. These will be called
	solution languages (see Definition \ref{def:sollanguages}). In Section
	\ref{sec:eqva} we show that the solution languages (with respect to a suitable
	generating set and normal form) are EDT0L. This will be a consequence of the
	stronger result that they are accepted by \emph{multivariable finite-state
	automata} (see Definition \ref{def:mFSA}):

	\begin{theorem_no_number}[\ref{theorem:VA_n_reg}]
		The solution language to any system of equations in virtually abelian group is accepted by a multivariable finite-state automaton.
	\end{theorem_no_number}

	\begin{cor_no_number}[\ref{VA_EDT0L_cor}]
		The solution language to any system of equations in a virtually
		abelian group is EDT0L.
	\end{cor_no_number}

%	Two basic properties of a language are its \emph{growth function}, $a_n$, which records the number of words of length $n$, and \emph{growth series}, $\sum_n a_nz_n$, the associated formal power series. It is natural to ask under what conditions a growth series is rational, i.e. contained in the ring of rational functions $\Q(z)$. In other words, when is it expressible as a ratio of polynomials in $z$ with rational (or, equivalently, integer) coefficients? More generally, a series may be algebraic over $\Q(z)$. A series is called \emph{transcendental} if it is not algebraic.

	Many of the papers that show that solution languages to systems of equations
	are EDT0L also consider the space complexity of the algorithms which
	constructs the EDT0L systems. We also show that both the multivariable
	finite-state automata and the EDT0L systems that accept these solution
	languages can be constructed in non-deterministic quadratic space (see Proposition \ref{prop:quadspace}).

	It is a standard fact that every regular language has rational growth series.
	That is, the generating function which counts the number of words in the
	language with increasing length lies in the ring of rational functions
	$\Q(z)$. A result of Chomsky and Sch\"utzenberger \cite{CS} asserts that the
	growth series of every unambiguous context-free language is algebraic over
	$\Q(z)$. In contrast, there is no reason to expect that those EDT0L languages
	which do not fall under these two cases have well-behaved growth series.
	Indeed, Corollary 8 of \cite{appl_L_systems_GT} implies that there are EDT0L
	languages with transcendental (i.e. non-algebraic) growth series. A priori,
	the language obtained in Corollary \ref{VA_EDT0L_cor} is neither regular nor
	context-free. Nevertheless, we prove that its growth series is rational.

	\begin{prop_no_number}[\ref{prop:EDT0Lrational}]
		The solution language to any system of equations in a virtually abelian group has rational growth series.
	\end{prop_no_number}

	Algebraic sets in groups can be seen as an analogue of the fundamental notion of algebraic varieties -- the zero-loci of systems of equations. Meuser \cite{Meuser}, and later Denef \cite{Denef}, proved the rationality of the Poincar\'e series of varieties over the $p$-adic integers, which can be thought of as a form of growth series. In Section \ref{sec:growth} we prove an analogous result for algebraic sets of virtually abelian groups, using a notion of growth appropriate to the setting of finitely generated groups, namely word growth. We will use the notion of a \emph{polyhedral set}, which has its roots in the model theory of Presburger (see Section \ref{sec:prelim} for definitions).

	Word growth in finitely generated groups is a much-studied topic. The growth function counts the number of group elements of length $n$, with respect to the metric arising from a choice of finite generating set. The asymptotics of this function are well understood, but many questions remain about the properties of the corresponding formal power series. For an introduction to the topic, the reader is directed to Mann's book \cite{Mann}.

	Any subset of a group has a growth function, inherited from the group itself.
	This \emph{relative growth} has been studied in many papers, including
	\cite{DO}. The relative growth series of any subgroup of a virtually abelian group
	was shown to be rational in \cite{Evetts}. In Section \ref{sec:growth} we
	consider the relative growth of the algebraic sets of a
	virtually abelian group, as sets of
	tuples of group elements (with an appropriate metric). We show that the growth series of an algebraic set
	is always a rational function, regardless of the choice of finite weighted generating set.

	\begin{theorem_no_number}[\ref{thm:soln_sets_rational}]
		Let $G$ be a virtually abelian group. Then every algebraic set of $G$ has rational weighted growth series with respect to any finite generating set.
	\end{theorem_no_number}

	Moreover, we consider the natural \emph{multivariate} growth series of the algebraic set, and demonstrate how recent results of Bishop imply that this series is holonomic (a class which includes algebraic functions and some transcendental functions).

	\begin{cor_no_number}[\ref{cor:holonomic}]
		Every algebraic set of a virtually abelian group has holonomic weighted multivariate growth series.
	\end{cor_no_number}

	We note that it may be useful for other purposes to have an explicit description of the algebraic sets of the groups in question, since this does not appear cleanly in the proofs. For such a statement, the interested reader is directed to Corollary \ref{rem:algstatement}, where the general structure of algebraic sets is noted, using the terminology of polyhedral sets.

\section{Preliminaries}
	\label{sec:prelim}
In this section we lay out the key definitions and basic results that will be required for the rest of the paper.

%	The following Lemma easily follows from standard results (for example
%	\cite{Mann}).
%
%	\begin{lem}\label{lem:charsubgroup}
%		If $G$ is a finitely generated virtually abelian group then it has a characteristic free abelian subgroup of finite index.
%	\end{lem}

 \begin{ntn}
	We will write functions to the right of their arguments, with the exception
	of growth functions and the generating functions of growth series.

	We will use \(\mathbb{Z}_{> 0}\) and \(\mathbb{Z}_{\geq 0}\) to denote the
	positive and non-negative integers, respectively.

	If $w\in S^*$ is a word in the generators of some group $G$, we write
	$\overline{w}\in G$ for the group element that the word $w$ represents.
\end{ntn}

\subsection{Polyhedral sets}
	Our fundamental tool for proving that languages of representatives have
	rational growth series in Proposition \ref{prop:EDT0Lrational} and Section
	\ref{sec:growth} will be the theory of \emph{polyhedral sets}. These ideas
	appear in model theory as early as Presburger \cite{Presburger}. Results
	regarding rationality can be found in \cite{Denef}, and the ideas also appear
	in the theory of Igusa local zeta functions (see \cite{Clucketal}). The
	following definitions and results follow Benson's work \cite{Benson}, where it
	is proved that virtually abelian groups have rational (standard) growth
	series. More recently, polyhedral sets have again been used to prove
	rationality of various growth series of groups
	(\cite{DuchinShapiro}, \cite{Evetts}).

	\begin{dfn}\label{def:polyhedral}
		Let $r\in\Z_{>0}$, and let $\cdot$ denote the Euclidean scalar product. Then we define the following.
		\begin{enumerate}
			\item Any subset of $\Z^r$ of the form $\{\mathbf{z}\in\Z^r\mid\mathbf{u}\cdot\mathbf{z}=a\}$, $\{\mathbf{z}\in\Z^r\mid\mathbf{u}\cdot\mathbf{z}>a\}$, or $\{\mathbf{z}\in\Z^r\mid\mathbf{u}\cdot\mathbf{z}\equiv a\mod b\}$, for any $\mathbf{u}\in\Z^r$, $a\in\Z$, $b\in\Z_{>0}$, will be called an \emph{elementary set};
			\item any finite intersection of elementary sets will be called a \emph{basic polyhedral set};
			\item any finite union of basic polyhedral sets will be called a \emph{polyhedral set}.
		\end{enumerate}
		If $\cP\subset\Z^r$ is polyhedral and additionally no element contains negative coordinate entries, we call $\cP$ a \emph{positive polyhedral set}.
	\end{dfn}
	It is not hard to prove the following closure properties.
	\begin{proposition}[Proposition 13.1 and Remark 13.2 of \cite{Benson}]\label{prop:polyclosed}
		Let $\cP, \ \cQ\subseteq\Z^r$ and $\cR\subseteq\Z^s$ be polyhedral sets for some positive integers $r$ and $s$. Then the following are also polyhedral: $\cP\cup\cQ\subseteq\Z^r$, $\cP\cap\cQ\subseteq\Z^r$, $\Z^r\setminus\cP$, $\cP\times\cR\subseteq\Z^{r+s}$.
	\end{proposition}
	Benson also shows that polyhedral sets behave well under affine transformations, as follows.
	\begin{dfn}
		We call a map $\mathcal{A}\colon\Z^r\to\Z^{s}$ an \emph{integral affine transformation} if there exists an $r\times s$ matrix $M$ with integer entries and some $\mathbf{q}\in\Z^{s}$ such that $\mathbf{p}\cA=\mathbf{p}M+\mathbf{q}$ for $\mathbf{p}\in\Z^r$.
	\end{dfn}
	\begin{proposition}[Propositions 13.7 and 13.8 of \cite{Benson}]\label{prop:polyaffine}
		Let $\cA$ be an integral affine transformation. If $\cP\subseteq\Z^r$ is a polyhedral set then $\cP\cA\subseteq\Z^{s}$ is a polyhedral set. If $\cQ\subseteq\Z^{s}$ is a polyhedral set then the preimage $\cQ\cA^{-1}\subseteq\Z^r$ is a polyhedral set.
	\end{proposition}
	We note that projection onto any subset of the coordinates of $\Z^r$ is an integral affine transformation.

	\begin{ntn}
		We will now introduce weight functions. When talking about weighted lengths
		of elements of free abelian or virtually abelian groups, we will use \(\|
		\cdot \|\) instead of \(|\cdot |\), which will be used for `standard' length
		of elements.
	\end{ntn}

	Let $\cP\subseteq\Z^r$ be a polyhedral set. Given some choice of
	weight function $\|\mathbf{e}_i\|\in\Z_{>0}$ for the standard basis vectors $\{\mathbf{e}_i\}_{i=1}^r$ of $\mathbb{Z}^r$, we assign the weight $\sum_{i=1}^r a_i\|\mathbf{e}_i\|$ to the element
	$(a_1, \ \ldots, \ a_r)\in \cP$. Define the spherical growth function
	\[\sigma_\cP(n)=\#\{\mathbf{p}\in\cP\mid\|\mathbf{p}\|=n\},\] and the resulting
	weighted growth series
	\[\mbS_\cP(z)=\sum_{n=0}^\infty\sigma_\cP(n)z^n.\] Our argument
	will rely on the following crucial proposition.
	\begin{proposition}[Proposition 14.1 of \cite{Benson}, and Lemma 7.5 of \cite{Denef}]\label{prop:polyhedralrationalgrowth}
		If $\cP$ is a positive polyhedral set, then the weighted growth series $\mbS_\cP(z)$ is a rational function of $z$.
	\end{proposition}

	We will need the following more general result.
	\begin{cor}\label{cor:rationalpoly}
		Let $\cP\subset\Z^r$ be any polyhedral set (not necessarily positive). Then the weighted growth series $\mbS_\cP(z)$ is a rational function of $z$.
	\end{cor}

	\begin{proof}
		We show that $\cP$ may be expressed as a disjoint union of polyhedral sets,
		each in weight-preserving bijection with a positive polyhedral set. Let
		$\cQ_1=\{\mathbf{z}\in\Z^r\mid \mathbf{z}\cdot \mathbf{e}_i\geq0,~1\leq i\leq
		r\}=\bigcap_{i=1}^r\{\mathbf{z}\in\Z^r\mid \mathbf{z}\cdot \mathbf{e}_i\geq 0\}$
		denote the non-negative orthant of $\Z^r$, and note that it is polyhedral.
		Let $\cQ_2, \ \ldots, \ \cQ_{2^r}$ denote the remaining orthants (in any
		order) obtained from $\cQ_1$ by (compositions of) reflections along
		hyperplanes perpendicular to the axes and passing through the origin. By
		Proposition \ref{prop:polyaffine} these are also polyhedral sets. Let
		$\cP_1=\cP\cap\cQ_1$ and for each $2\leq j\leq 2^r$, inductively define
		\[\cP_j = \left(\cP\setminus\bigcup_{k<j}\cP_k\right) \cap \cQ_j.\] Each
		$\cP_j$ is a polyhedral set by Proposition \ref{prop:polyclosed}, and we
		have a disjoint union $\cP=\bigcup_{j=1}^{2^r}\cP_j$. Each $\cP_j$ is in
		weight-preserving bijection with a \emph{positive} polyhedral set (by
		compositions of reflections along hyperplanes) and so $\mbS_{\cP_j}(z)$ is
		rational. The result follows since
		$\mbS_{\cP}(z)=\sum_{j=1}^{2^r}\mbS_{\cP_j}(z)$.
	\end{proof}

\subsection{Equations}

  \begin{dfn}
    Let \(G\) be a group. A \textit{finite system of equations} in \(G\) is a
    finite subset \(\mathcal E\) of \(G \ast F_V\), where \(F_V\) is the free
    group on a finite set \(V\). If \(\mathcal E = \{w_1, \ \ldots, \ w_n\}\),
    we denote the system \(\mathcal E\) by \(w_1 = w_2 = \cdots = w_n = 1\).
    Elements of \(V\) are called \textit{variables}. A \textit{solution} to a
    system \(w_1 = \cdots = w_n = 1\) is a homomorphism \(\phi \colon F_V \to
    G\), such that \(w_1 \bar{\phi} = \cdots = w_n \bar{\phi} = 1_G\), where
    \(\bar{\phi}\) is the extension of \(\phi\) to a homomorphism from \(G \ast
    F_V \to G\), defined by \(g \bar{\phi} = g\) for all \(g \in G\).

    % A \textit{finite system of inequations} in \(G\) is also a finite subset
    % \(E\) of \(G \ast F_V\), for some finite set \(X\), and is denoted
    % \(\omega_1 \neq 1, \ \omega_2 \neq 1, \ldots, \ \omega_n \neq 1\). Again,
    % elements of \(X\) are called \textit{variables}, and a \textit{solution} is
    % a homomorphism \(\phi \colon F_V \to G\), such that \(\omega_1 \bar{\phi}
    % \neq 1_G, \ \ldots, \ \omega_n \bar{\phi} \neq 1_G\).

    A \textit{finite system of twisted equations} in \(G\) is a finite subset
    \(\mathcal E\) of \(G \ast (F_V \times \Aut(G))\), and is again denoted
    \(w_1 = \cdots = w_n = 1\). Elements of \(V\) are
    called \textit{variables}. Define the function
    \begin{align*}
      p \colon G \times \Aut(G) & \to G \\
      (g, \ \psi) & \mapsto g \psi.
    \end{align*}
    If \(\phi \colon F_V \to G\) is a homomorphism, let \(\tilde{\phi}\) denote
    the homomorphism from \(G \ast (F_V \times \Aut(G))\) to \(G \times \Aut(G)\),
    defined by \((h, \ \psi) \tilde{\phi} = (h \phi, \ \psi)\) for \((h, \ \psi)
    \in F_V \times \Aut(G)\) and \(g \tilde{\phi} = g\) for all \(g \in G\). A
    \textit{solution} is a homomorphism \(\phi \colon F_V \to G\), such that
    \(w_1 \tilde{\phi} p = \cdots = w_n \tilde{\phi} p = 1_G\).

		For the purposes of decidability, in finitely generated groups, the elements
		of \(G\) will be represented as words over a finite generating set, and in
		twisted equations, automorphisms will be represented by their action on the
		generators.
  \end{dfn}

	\begin{rmk}
		We will often display a solution to an equation with variables \(X_1, \
		\ldots, \ X_n\) as a tuple \((x_1, \ \ldots, \ x_n)\) of group elements,
		rather than a homomorphism. We can obtain the homomorphism from the tuple by
		defining \(X_i \mapsto x_i\) for each \(i\).
	\end{rmk}

	\begin{dfn}
		The set of solutions to a finite system of equations in $n$ variables, expressed as a subset of $G^n$, is called an \emph{algebraic set} of $G$.
	\end{dfn}
\subsection{Multivariable finite-state automata}

	Since solutions to equations can be thought of as tuples, one method that can
	be used to study the language complexity of sets of solutions is using
	multivariable languages, which are sets of tuples of words over an alphabet.
	We start with the formal definition.

	\begin{dfn}\label{def:mFSA}
		Let \(\Sigma\) be an alphabet, and \(n \in \mathbb{Z}_{>0}\). An
		\textit{\(n\)-variable word} over \(\Sigma\) is an element of the Cartesian
		product \((\Sigma^\ast)^n\), and an \textit{\(n\)-variable language} over
		\(\Sigma\) is any subset of \((\Sigma^\ast)^n\).
	\end{dfn}

	We continue with a generalisation of a finite-state automaton to accept
	\(n\)-variable languages, for some positive integer \(n\):
	the (asynchronous, non-deterministic) \(n\)-variable finite-state automaton.

	\begin{dfn}
		Let \(n \in \mathbb{Z}_{>0}\). An \textit{\(n\)-variable finite-state
		automaton} is a tuple \(\mathcal{A} = (\Sigma, \ \Gamma, \ q_0, \ F)\), where
		\begin{enumerate}
			\item \(\Sigma\) is an alphabet;
			\item \(\Gamma\) is a finite edge-labelled graph, where labels are
			\(n\)-variable words in \((\Sigma^\ast)^n\), with at most one non-empty
			word entry. The vertices of \(\Gamma\)
			are called \textit{states};
			\item \(q_0 \in V(\Gamma)\) is called the \textit{start state};
			\item \(F \subseteq V(\Gamma)\) is called the set of \textit{accept
			states}.
		\end{enumerate}
		When tracing a path in \(\Gamma\), we trace an \(n\)-variable word to be the
		concatenation of the labels of each edge traversed. Since each edge has at
		most one non-empty entry, the word will only get longer in one coordinate at
		a time. An \(n\)-variable word \(\mathbf{w} \in (\Sigma^\ast)^n\) is
		\textit{accepted} by \(\mathcal{A}\) if there is a path \(\gamma\) in
		\(\Gamma\) from \(q_0\) to a state in \(F\), such that the \(n\)-variable
		word obtained by reading the labels in \(\gamma\) is \(\mathbf{w}\). The
		\textit{language accepted} by \(\mathcal{A}\) is the set of all
		\(n\)-variable words accepted by \(\mathcal{A}\).
	\end{dfn}

	We give an example of a language accepted by a \(3\)-variable finite-state
	automaton. In this case, the language represents the set of solutions to a
	system of equations in \(\mathbb{Z}\).

	\begin{ex}
		\label{exp_lin_n_reg_ex}
		Let \(\mathcal{E}\) be the following system of equations in \(\mathbb{Z}\)
		(using additive notation):
		\begin{align*}
			& X - Y + Z = 1
			& - Y + Z = 0.
		\end{align*}
		Note that by subtracting the second equation from the first, it is not
		difficult to show that the set of solutions to this system is \(\{(1, \ y, \
		y) \mid y \in \mathbb{Z}\}\). To demonstrate a more general method we will
		use later on, we will construct the set of solutions, and show that \(L =
		\{(a^x, \ a^y, \ a^z) \mid (x, \ y, \ z) \text{ is a solution to }
		\mathcal{E}\}\) is accepted by a \(3\)-variable finite-state automaton over
		the alphabet \(\{a, \ a^{-1}\}\), using a different method. We will show
		\begin{enumerate}
			\item The language \(\{(a^x, \ a^y, \ a^z) \mid (x, \ y, \ z) \text{ is
			a solution to } \mathcal{E} \text{ and } x, \ y, \ z \geq 0\}\) is accepted
			by a \(3\)-variable finite-state automaton;
			\item To show \(L\) is accepted by a \(3\)-variable finite-state
			automaton, we take the finite union across the possible configurations of
			signs of \(X\), \(Y\) and \(Z\) and use the fact that finite unions of
			languages accepted by \(3\)-variable finite-state automata are also
			accepted by \(3\)-variable finite-state automata.
		\end{enumerate}
		To show (1), consider the \(3\)-variable
		finite-state automaton in Figure 1.
		\begin{figure}
			\caption{The start state is \(q_{(0, \ 0)}\), and
			\(q_{(1, \ 0)}\) is the unique accept state.}
			\begin{tikzpicture}
				[scale=.7, auto=left,every node/.style={circle}]
				\tikzset{
				% style to apply some styles to each segment of a path
				on each segment/.style={
					decorate,
					decoration={
						show path construction,
						moveto code={},
						lineto code={
							\path [#1]
							(\tikzinputsegmentfirst) -- (\tikzinputsegmentlast);
						},
						curveto code={
							\path [#1] (\tikzinputsegmentfirst)
							.. controls
							(\tikzinputsegmentsupporta) and (\tikzinputsegmentsupportb)
							..
							(\tikzinputsegmentlast);
						},
						closepath code={
							\path [#1]
							(\tikzinputsegmentfirst) -- (\tikzinputsegmentlast);
						},
					},
				},
				% style to add an arrow in the middle of a path
				mid arrow/.style={postaction={decorate,decoration={
							markings,
							mark=at position .5 with {\arrow[#1]{stealth}}
						}}},
			}

				\node[draw, minimum size=1.8cm] (m1m1) at (-5, -5) {\(q_{(-1, \ -1)}\)};
				\node[draw, minimum size=1.8cm] (m10) at (-5, 0) {\(q_{(-1, \ 0)}\)};
				\node[draw, minimum size=1.8cm] (m11) at (-5, 5) {\(q_{(-1, \ 1)}\)};
				\node[draw, minimum size=1.8cm] (0m1) at (0, -5) {\(q_{(0, \ -1)}\)};
				\node[draw, minimum size=1.8cm] (00) at (0, 0)  {\(q_{(0, \ 0)}\)};
				\node[draw, minimum size=1.8cm] (01) at (0, 5) {\(q_{(0, \ 1)}\)};
				\node[draw, minimum size=1.8cm] (1m1) at (5, -5) {\(q_{(1, \ - 1)}\)};
				\node[draw, minimum size=1.8cm, double] (10) at (5, 0)
				{\(q_{(1, \ 0)}\)};
				\node[draw, minimum size=1.8cm] (11) at (5, 5) {\(q_{(1, \ 1)}\)};

				\draw[postaction={on each segment={mid arrow}}] (m1m1) to (0m1);
				\draw[postaction={on each segment={mid arrow}}] (m10) to (00);
				\draw[postaction={on each segment={mid arrow}}] (m11) to (01);
				\draw[postaction={on each segment={mid arrow}}] (0m1) to (1m1);
				\draw[postaction={on each segment={mid arrow}}] (00) to (10);
				\draw[postaction={on each segment={mid arrow}}] (01) to (11);

				\draw[postaction={on each segment={mid arrow}}, out=240, in=30,
				distance=2cm] (00) to (m1m1);
				\draw[postaction={on each segment={mid arrow}}, out=240, in=30,
				distance=2cm] (01) to (m10);
				\draw[postaction={on each segment={mid arrow}}, out=240, in=30,
				distance=2cm] (10) to (0m1);
				\draw[postaction={on each segment={mid arrow}}, out=240, in=30,
				distance=2cm] (11) to (00);

				\draw[postaction={on each segment={mid arrow}}, out=60, in=210,
				distance=2cm] (m1m1) to (00);
				\draw[postaction={on each segment={mid arrow}}, out=60, in=210,
				distance=2cm] (m10) to (01);
				\draw[postaction={on each segment={mid arrow}}, out=60, in=210,
				distance=2cm] (0m1) to (10);
				\draw[postaction={on each segment={mid arrow}}, out=60, in=210,
				distance=2cm] (00) to (11);

				\node at (-2.5, 5.4) {\((a, \ \varepsilon, \ \varepsilon)\)};
				\node at (2.5, 5.4) {\((a, \ \varepsilon, \ \varepsilon)\)};
				\node at (-2.5, 0.4) {\((a, \ \varepsilon, \ \varepsilon)\)};
				\node at (2.5, 0.4) {\((a, \ \varepsilon, \ \varepsilon)\)};
				\node at (-2.5, -4.6) {\((a, \ \varepsilon, \ \varepsilon)\)};
				\node at (2.5, -4.6) {\((a, \ \varepsilon, \ \varepsilon)\)};

				\node at (-0.7, 2) {\((\varepsilon, \ a, \ \varepsilon)\)};
				\node at (4.3, 2) {\((\varepsilon, \ a, \ \varepsilon)\)};
				\node at (-0.7, -3) {\((\varepsilon, \ a, \ \varepsilon)\)};
				\node at (4.3, -3) {\((\varepsilon, \ a, \ \varepsilon)\)};

				\node at (-4.3, 3) {\((\varepsilon, \ \varepsilon, \ a)\)};
				\node at (0.7, 3) {\((\varepsilon, \ \varepsilon, \ a)\)};
				\node at (-4.3, -2) {\((\varepsilon, \ \varepsilon, \ a)\)};
				\node at (0.7, -2) {\((\varepsilon, \ \varepsilon, \ a)\)};

			\end{tikzpicture}
		\end{figure}
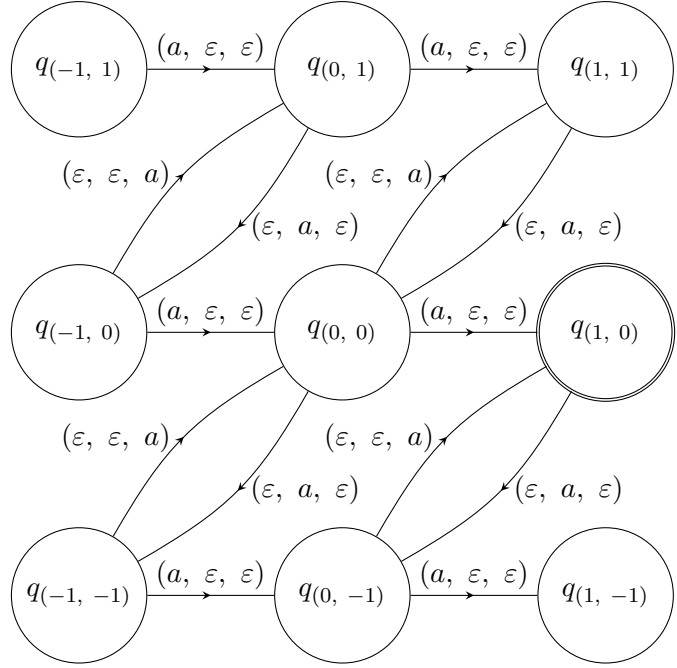
		This finite-state automaton works as follows:
		\begin{enumerate}
			\item Traversing an edge labelled by \((a, \ \varepsilon, \
			\varepsilon)\), \((\varepsilon, \ a, \ \varepsilon)\) or \((\varepsilon, \
			\varepsilon, \ a)\) corresponds to increasing \(x\), \(y\) or \(z\) by
			\(1\), respectively. The states \(q_{(i, \ j)}\) correspond to the value
			of \((x - y + z, \ -y + z)\), with the current values of \(x\), \(y\) and
			\(z\).
			\item Once we have increased \(x\), \(y\) and \(z\) to the desired values,
			if this is a solution to \(\mathcal{E}\), then we must be in the accept
			state \(q_{(1, \ 0)}\).
			\item Note that we cannot increase the \(x\)s, \(y\)s and \(z\)s in any
			order, otherwise we would need an unbounded size of FSA. For example, the
			element \((a, \ a^l, \ a^l)\) of \(L\), where \(l \in \mathbb{Z}\) and \(l
			> 1\) cannot be reached in the above system by traversing one
			\((a, \ \varepsilon, \ \varepsilon)\) edge, then \(l\) \((\varepsilon, \
			a, \ \varepsilon)\) edges, and then \(l\) \((\varepsilon, \ \varepsilon, \
			a)\) edges, as after the \(l\) \((\varepsilon, \ a, \ \varepsilon)\) edges
			we would need a state \(q_{(-l + 1, \ -l)}\), which does not lie in the
			finite-state automaton. Moreover, we cannot add them to the finite-state
			automaton, as there are infinitely many such states. We prove the
			existence of an ordering of the edges (up to considering two edges with the
			same label equivalent) that works in Lemma \ref{eqn_bound_K_lem}. In this
			specific case, it is not hard to show that the ordering that starts with
			\((a, \ \varepsilon, \ \varepsilon)\), followed by \(l\) traversals of a
			path comprising one \((\varepsilon, \ a, \ \varepsilon)\) edge and one
			\((\varepsilon, \ \varepsilon, \ a)\) edge, for all \(l > 0\), and a
			similar `reversed' order would work if \(l < 0\).
			\item Note that not all states may be necessary, but it is simpler to construct them all.
		\end{enumerate}
	\end{ex}

\subsection{EDT0L languages}

	The ultimate aim of using the \(n\)-variable finite-state automata is in order
	to show that the set of solutions to a system of equations in a virtually
	abelian group can be described as an EDT0L language, thus adding an additional
	class of groups to many classes of groups already known to have this property.
	We start by defining an EDT0L system; a grammar that generates an EDT0L
	language.

  \begin{dfn}
    An \textit{EDT0L system} is a tuple \(\mathcal H = (\Sigma, \ C, \ w, \
    \mathcal{R})\), where
    \begin{enumerate}
      \item \(\Sigma\) is an alphabet, called the \textit{(terminal) alphabet};
      \item \(C\) is a finite superset of \(\Sigma\), called the
      \textit{extended alphabet} of \(\mathcal H\);
      \item \(w \in C^\ast\) is called the \textit{start word};
      \item \(\mathcal{R}\) is a regular (as a language) subset of
      \(\End(C^\ast)\), called the \textit{rational control} of \(\mathcal H\).
    \end{enumerate}
    The language \textit{accepted} by \(\mathcal H\) is
    \[
      L(\mathcal H) = \{w \phi \mid \phi \in \mathcal{R}\}.
    \]
    A language accepted by an EDT0L system is called an \textit{EDT0L
    language}.
  \end{dfn}

	There are a number of different definitions of an EDT0L system, that all
	generate the same class of languages. In \cite{eqns_free_grps} and
	\cite{appl_L_systems_GT}, the definition is the same as given here, except for
	the insistence that the start word is a single letter. To show that these are
	equivalent, adding a single homomorphism preconcatenated to the rational
	control of an EDT0L system (as defined here) that maps a new letter \(\perp\)
	to the start word, and then defining the new start symbol to be \(\perp\)
	gives that any EDT0L language is accepted by an EDT0L system with a single
	letter as the start word. In \cite{math_theory_L_systems}, and many earlier
	publications, the definition is what is given above, except they only allow
	rational controls of the form \(\Delta^\ast\), for some finite set of
	endomorphisms \(\Delta\). This definition is again equivalent to the
	definition we have given \cite{Asveld}, but proves to be cumbersome when
	proving languages are EDT0L.

	To streamline the definition of specific EDT0L systems, we introduce the
	following notation convention for specifing endomorphisms of a given
	free monoid.

	\begin{ntn}
		When defining endomorphisms of \(C^\ast\) for some extended alphabet \(C\),
		within the definition of an EDT0L system, we will usually define each
		endomorphism by where it maps each letter in \(C\). If any letter is not
		assigned an image within the definition of an endomorphism, we will say that
		it is fixed by that endomorphism.
	\end{ntn}

	\subsection{Space complexity}

		We give a brief definition of space complexity. We refer the reader to
		\cite{computational_compl} for a comprehensive introduction to space
		complexity, or to \cite{eqns_free_grps} for the consideration of space
		complexity when constructing EDT0L systems.

		\begin{dfn}
			Let \(f \colon \mathbb{Z}_{\geq 0} \to \mathbb{Z}_{\geq 0}\) be a function.
			We say that an algorithm runs in \(\mathsf{NSPACE}(f)\) if it can be
			performed by a non-deterministic Turing machine with the following:
			\begin{enumerate}
				\item A read-only input tape;
				\item A write-only output tape;
				\item A read-write work tape that has a length of at most
				\(\mathcal{O}(nf)\) for an input of length \(n\).
			\end{enumerate}

			An algorithm is said to run in \textit{non-deterministic quadratic space} if
			it runs in \(\mathsf{NSPACE}(f)\), for some quadratic function \(f \colon
			\mathbb{Z}_{\geq 0} \to \mathbb{Z}_{\geq 0}\).
		\end{dfn}

		We will use this definition to show that we can construct the multivariable
		finite-state automaton from Theorem \ref{theorem:VA_n_reg}, and hence the
		EDT0L system from Corollary \ref{VA_EDT0L_cor}, in non-deterministic
		quadratic space. Recall that a multivariable finite-state automaton has a
		set of vertices, edges, an assignment of labels to edges, a specified start
		state, and a set of accept states that all must be constructed.

		We will later need the following lemma that allows us to take finite
		unions of languages that are accepted by multivariable finite-state
		automata without changing the space complexity.

		\begin{lem}
			\label{finite_union_n_reg_space_complexity_lem}
			Let \(f \colon \mathbb{Z}_{\geq 0} \to \mathbb{Z}_{\geq 0}\) be a
			function. A finite union of languages accepted by multivariable
			finite-state automata that are all constructible in \(\mathsf{NSPACE}(f)\)
			is also accepted by a multivariable finite-state automaton that is
			constructible in \(\mathsf{NSPACE}(f)\).
		\end{lem}

		\begin{proof}
			The automaton \(\mathcal{M}\) we use is the automaton obtained by taking
			the union of all of the automata of the languages in the union, and
			collapsing the start states to a single state, which will be the start
			state. All accept states will remain accept states. We can construct
			\(\mathcal{M}\) by constructing each of the automata in the union one at
			a time, which can be done in \(\mathsf{NSPACE}(f)\).
		\end{proof}

\subsection{Languages of solutions to equations}

	We now define the languages that we will be studying, which are derived
	from the set of solutions. We need to choose a finite generating set and
	normal form in order to do this, although any can work. We consider
	two methods: one can either look at the language of
	\(n\)-variable words representing solutions, or one can look at the language
	of words that comprise the solutions concatenated with one another, delimited
	by an additional letter \(\#\).

	\begin{dfn}\label{def:sollanguages}
		Let \(G\) be a finitely generated group, with a finite monoid generating set
		\(\Sigma\), and a normal form \(\eta \colon G \to \Sigma^\ast\). Let
		\(\mathcal{E}\) be a system of (twisted) equations in \(G\), and let \(n\)
		be the number of variables in \(\mathcal{E}\). Let \(V =\{X_1, \ \ldots, \
		X_n\}\) be the set of variables, and let \(\mathcal{S}\) be the set of
		solutions, which are homomorphisms from \(F_V \ast G\) to \(G\).

		The \textit{multivariable solution language} to \(\mathcal{E}\) with
		respect to \(\Sigma\) and \(\eta\), is defined to be
		\[
			\{(X_1 \psi \eta, \ X_2 \psi \eta, \ \ldots, \ X_n \psi \eta) \mid \psi
			\in \mathcal{S}\}\subset \Sigma^*\times\Sigma^*\times\cdots\times\Sigma^*.
		\]
		The \textit{\(\#\)-joined solution language} to \(\mathcal{E}\) with respect
		to \(\Sigma\) and \(\eta\), is defined to be
		\[
			\{X_1 \psi \eta \# X_2 \psi \eta \# \cdots X_n \psi \eta \mid \psi \in
			\mathcal{S}\}\subset (\Sigma\cup\{\#\})^*.
		\]
	\end{dfn}

	We now show that a multivariable solution language being accepted by an
	\(n\)-variable finite-state automaton is sufficient for the corresponding \(\#\)-joined solution language to be EDT0L.

	\begin{lem}
		\label{n_reg_implies_EDT0L_lem}
		Let \(L\) be an \(n\)-variable language over an alphabet \(\Sigma\) (where \(n \in
		\mathbb{Z}_{>0}\)), that is
		accepted by an \(n\)-variable finite-state automaton, constructible in \(\mathsf{NSPACE}(f)\),
		for some \(f \colon \mathbb{Z}_{\geq 0}
		\to \mathbb{Z}_{\geq 0}\). Then
		\begin{enumerate}
			\item The language \(M = \{w_1 \# \cdots \# w_n \mid (w_1, \ \ldots, \ w_n) \in L\}\) is an
			EDT0L language over \(\Sigma \sqcup \{\#\}\);
			\item An EDT0L system for \(M\) can be
			constructed in \(\mathsf{NSPACE}(f)\).
		\end{enumerate}
	\end{lem}

	\begin{proof}
		We will construct an EDT0L system \(\mathcal{H}\) for \(M\) as follows. The
		terminal alphabet will be \(\Sigma \cup \{\#\}\), the extended alphabet will
		be \(C = \Sigma \cup \{\#, \ \perp_1, \ \ldots, \ \perp_n\}\), and the start
		word will be \(\perp_1 \# \cdots \# \perp_n\).

		Let \(\mathcal{A} = (\Sigma, \ \Gamma, \ q_0, \ F)\) be an \(n\)-variable
		finite-state automaton that accepts \(L\). We will use \(\mathcal{A}\) to
		define the rational control of \(\mathcal{H}\). Let \(W\) be the set of all
		\(n\)-variable words that appear as edge labels within \(\Gamma\). For each
		\(\mathbf{w} = (w_1, \ \ldots, \ w_n) \in W\), define \(\varphi_\mathbf{w}
		\in \End(C^\ast)\) by
		\begin{align*}
			\perp_1 \varphi_\mathbf{w} & = w_1 \perp_1 \\
			& \vdots \\
			\perp_n \varphi_\mathbf{w} & = w_n \perp_n.
		\end{align*}
		Also define \(\psi \in \End(C^\ast)\)
		\[
			\perp_i \psi = \varepsilon,
		\]
		for all \(i \in \{1, \ \ldots, \ n\}\). Our rational control \(\mathcal{R}\)
		will be a regular language over the set \(\{\varphi_\mathbf{w} \mid
		\mathbf{w} \in W\}\). Let \(\Gamma'\) be the edge-labelled graph obtained
		from \(\Gamma\), by replacing the label \(\mathbf{w}\) on each edge with
		\(\varphi_\mathbf{w}\).

		Consider the (\(1\)-variable) finite-state automaton \(\mathcal{B} =
		(\Sigma, \ \Gamma', \ q_0, \ F)\). Let \(K\) be the language accepted by
		\(\mathcal{B}\). We have that \(K\) is precisely the set of all
		endomorphisms \(\theta\) of \(C^\ast\) that can be written as products of
		endomorphisms \(\varphi_\mathbf{w}\), for \(\mathbf{w} \in W\), such that
		\((\perp_1 \# \cdots \# \perp_n) \theta = u_1 \perp_1 \# \cdots \# u_n
		\perp_n\), for some \((u_1, \ \ldots, \ u_n) \in L\). Therefore, the regular
		language \(K \psi\) is the set of all endomorphisms that map \(\perp_1 \#
		\cdots \# \perp_n\) to an element of \(M\), and so taking \(\mathcal{R} = K
		\psi\) gives the desired EDT0L system.

		For (2), since a multivariable finite state automaton contains an alphabet
		\(\Sigma\), this can be obtained and output in \(\mathsf{NSPACE}(f)\), and
		thus the alphabet for the EDT0L language, \(\Sigma \cup \{\#\}\), and the
		extended alphabet \(C = \Sigma \cup \{\#, \ \perp_1, \ \ldots, \ \perp_n\}\)
		can also be constructed and written to the output tape in
		\(\mathsf{NSPACE}(f)\). The start word will always be \(\perp_1 \# \cdots \#
		\perp_n\), regardless of the input, and we can just output this.

		It remains to construct the rational control. As in the construction of
		\(\mathcal{H}\), we use the same set of vertices and edges, but whenever the
		rational control in \(\mathcal{H}\) labels an edge using
		\(\varphi_\mathbf{w}\), we instead label it using \(\mathbf{w}\), and note
		that \(\varphi_{\mathbf{w}}\) can be effectively computed from
		\(\mathbf{w}\). To record \(\varphi_{\mathbf{w}}\), we only need to know
		where each \(\perp_i\) maps (as they always fix everything else), and that
		is precisely the information that \(\mathbf{w}\) contains.
	\end{proof}

\section{Solution languages in virtually abelian groups}
	\label{sec:eqva}
	The purpose of this section is to prove that the multivariable solution
	languages to systems of equations in virtually abelian groups are accepted by
	multivariable finite-state automata, and so \(\#\)-joined solution languages
	are EDT0L, all with respect to a natural generating set and normal form. We
	do this by first showing that the multivariable solution languages for systems
	of twisted equations in free abelian groups are recognised by finite-state
	automata, and then prove that equations in virtually abelian groups reduce to
	twisted equations in free abelian groups. Throughout this section, when
	referring to free abelian groups, we will use additive notation. This means
	that equations in free abelian groups will be expressed as sums rather than
	`products'. When representing solution languages, we will express them using
	multiplicative notation, as this is more natural with languages, using \(a_1,
	\ \ldots, \ a_k\) to be the standard generators of \(\mathbb{Z}^k\).

	The next lemmas are used to prove that systems of equations, and therefore twisted
	equations, in free abelian groups have multivariable solution languages
	accepted by \(n\)-variable finite-state automata, where \(n\) is the number of
	variables. The fact that free abelian groups have EDT0L \(\#\)-joined languages is
	already known; Diekert, Je\.{z} and Kufleitner \cite{EDT0L_RAAGs} show that
	right-angled Artin groups have EDT0L \(\#\)-joined languages, and Diekert
	\cite{more_than_1700} has a more direct method for systems of equations in
	\(\mathbb{Z}\), which can easily be generalised to all free abelian groups.
	For the sake of completeness, we give our own argument here.
	
	We begin with the following technical definition.

	\begin{dfn}
		Let $B=[b_{ij}]$ be an $n\times m$ integer matrix. Then define a function $|\cdot|_B\colon\R^n\to\R$ via
		\[|(y_1,\ldots,y_n)|_B = \max\left(\left\lvert\sum_{i=1}^n y_ib_{i1}\right\rvert,\ \left\lvert\sum_{i=1}^n y_ib_{i2}\right\rvert,\ \ldots,\ \left\lvert\sum_{i=1}^n y_ib_{im}\right\rvert\right).\]
		In other words, $|\mathbf{y}|_B$ is the maximal absolute value of the coordinates of the vector $\mathbf{y}B$.
		
		Note that if $\mathbf{y}\in\Z^n$ then $|\mathbf{y}|_B\in\Z$, and that $|\cdot|_B$ satisfies the triangle inequality.
	\end{dfn}
	We now show that we can construct any solution to a system of equations while controlling the value of $|\cdot|_B$ at each intermediate point.
	\begin{lem}
		\label{eqn_bound_K_lem}
		Let $B$ be an $n\times m$ integer matrix, $X$ be a vector of $n$ variables,
		$\mathbf{c}\in\Z^m$, and consider the system of $n$ equations over $\Z$ given by
		$\mathbf{X}B=\mathbf{c}$. Write $b_{\max}=\max_{i,j}|b_{ij}|$ and let $K=
		\max(|c_1|,\ \ldots,\ |c_m|) + n^{3/2}\cdot b_{\max}$.

		Then, for each $\mathbf{x}\in\Z^n$ such that $\mathbf{x}B=\mathbf{c}$, there is a sequence \[\{0=\mathbf{x}^{(0)}, \ \mathbf{x}^{(1)}, \ \ldots, \ \mathbf{x}^{(k)} =\mathbf{x} \}\subset\Z^n\] with each $\mathbf{x}^{(j)}=\mathbf{x}^{(j-1)}+\mathbf{e}_j$ for some positive or negative standard basis vector $\mathbf{e}_j$, such that $|\mathbf{x}^{(j)}|_B\leq K$ for each $j\in\{1,\ \ldots,\ k\}$.
	\end{lem}
	\begin{proof}
	First, consider the straight line segment $L\subset\R^n$ from $0$ to $\mathbf{x}$. Since $B$ defines a linear transformation $\R^n\to\R^m$, the function $|\cdot|_B\colon\R^n\to\R$ is monotone non-decreasing as we move along $L$ from $0$ to $\mathbf{x}$. Therefore, for each $\mathbf{y}\in L$, we have  $|\mathbf{y}|_B\leq|\mathbf{x}|_B=\max(|c_1|,\ \ldots,\ |c_m|)$. To obtain the required sequence, we approximate $L$ with a piecewise linear path comprised of (positive and negative) standard basis vectors.

	Consider the set of unit $n$-cubes with integer-valued corners, which intersect $L$. From among the corners of these cubes, we can find a sequence $\{\mathbf{x}^{(j)}\}\subset\Z^n$ of integer-valued points, where subsequent terms share a cube edge (and so each $\mathbf{x}^{(j)}=\mathbf{x}^{(j-1)}+\mathbf{e}_j$ for some $\mathbf{e}_j$), such that $\mathbf{x}^{(0)}=0$ and $\mathbf{x}^{(k)}=\mathbf{x}$, for some $k$. We will show that each point in this sequence satisfies the required bound.
	
	Since the diameter of a unit $n$-cube is $\sqrt{n}$, each point $\mathbf{x}^{(j)}$ is a Euclidean distance of at most $\sqrt{n}$ from the line $L$. In other words, for each $j$ we have $\mathbf{x}^{(j)}=\mathbf{y}+\mathbf{d}$ for some $\mathbf{y}\in L$ and $\mathbf{d}=(d_1,\ldots,d_n)\in\R^n$ with $|d_i|\leq\sqrt{n}$. Then note that for any such $\mathbf{d}$ we have
	\begin{align*}
		|\mathbf{d}|_B &= \max\left(\left\lvert\sum_{i=1}^n d_ib_{i1}\right\rvert, \ \ldots, \ \left\lvert \sum_{i=1}^n d_ib_{im}\right\rvert \right)\\
		&\leq \max\left(\sum_{i=1}^n|d_i||b_{i1}|, \ \ldots, \ \sum_{i=1}^n |d_i||b_{im}| \right) \\
		&\leq \max\left( \sum_{i=1}^n\sqrt{n}\cdot b_{\max}, \ \ldots, \ \sum_{i=1}^n\sqrt{n}\cdot b_{\max}\right)  = (n\sqrt{n}) b_{\max}.
	\end{align*}
	We can then bound each element of the sequence as follows:
	\begin{align*}
		|\mathbf{x}^{(j)}|_B = |\mathbf{y}+\mathbf{d}|_B\leq |\mathbf{y}|_B + |\mathbf{d}|_B \leq \max(|c_1|,\ \ldots,\ |c_m|) + (n\sqrt{n})b_{\max}=K.
	\end{align*}
	Thus the sequence $\{\mathbf{x}^{(j)}\}$ satisfies the requirements of the Lemma.
	\end{proof}

	We now show that a system of twisted equations in $\Z^k$ can be reduced to a system of (non-twisted) equations in $\Z$.
	\begin{lem}\label{lem:Zsols}
		Let \(\mathcal{S}_\mathcal{E}\) be the solution set of a finite system
		\(\mathcal{E}\) of twisted equations in \(\mathbb{Z}^k\) in \(n\) variables.
		Then there is a finite system of equations \(\mathcal{F}\) in \(\mathbb{Z}\)
		with \(kn\) variables and solution set \(\mathcal{S}_\mathcal{F}\) such that
		\[
			\mathcal{S}_\mathcal{E} = \{((x_1, \ \ldots, \ x_k), \ (x_{k + 1}, \
			\ldots, \ x_{2k}), \ \ldots, \ (x_{(k - 1)n}, \ \ldots, \ x_{kn})) \mid
			(x_1, \ \ldots, \ x_{kn}) \in \mathcal{S}_\mathcal{F}\}.
		\]
	\end{lem}

	\begin{proof}
		Consider a twisted equation	in \(\mathbb{Z}^k\)
		\begin{align}
			\label{Zk_twisted_init_eqn}
			\mathbf{c}_0 + \mathbf{Y}_{i_1} \Phi_1 + \mathbf{c}_1 \cdots +
			\mathbf{Y}_{i_n} \Phi_n + \mathbf{c}_n = 0,
		\end{align}
 		where \(\mathbf{Y}_1, \ \ldots, \ \mathbf{Y}_n\) are variables,
 		\(\mathbf{c}_0, \ \ldots, \ \mathbf{c}_n \in \mathbb{Z}^n\) are constants,
 		and \(\Phi_1, \ \ldots, \ \Phi_n \in \GL_k(\mathbb{Z})\). Set \(\mathbf{c}
 		= \mathbf{c}_0 + \cdots + \mathbf{c}_n\). By grouping the occurences of
		each \(\mathbf{Y}_i\), we have that (\ref{Zk_twisted_init_eqn}) is
		equivalent to the following identity
		\begin{align}
			\label{Zk_pseudo_twisted_eqn}
				\mathbf{Y}_1 B_1 + \cdots \mathbf{Y}_n B_n + \mathbf{c} = 0,
		\end{align}
		where \(B_1 = [b_{1ij}], \ \ldots, \ B_n = [b_{nij}]\) are \(k \times k\)
		integer-valued matrices, although not necessarily in \(\GL_k(\mathbb{Z})\).
		We will first show that the solution set of	(\ref{Zk_pseudo_twisted_eqn}) is
		equal to the solution set of a system of \(k\) equations in \(\mathbb{Z}\).
		Write \(\mathbf{Y}_i = (Y_{i1}, \ \ldots, \ Y_{ik})\) and \(\mathbf{c} =
		(c_1, \ \ldots, \ c_n)\) for variables \(Y_{ij}\) over \(\mathbb{Z}\) and
		constants \(c_i \in \mathbb{Z}\), for each \(i\). Then \(\mathbf{Y}_i B_i =
		\left(\sum_{j = 1}^k b_{ij1} Y_{ij}, \ \ldots, \ \sum_{j = 1}^k b_{ijk}
		Y_{ij} \right)\), for each \(i\). It follows that the solution set of
		(\ref{Zk_pseudo_twisted_eqn}) is equal to the solution set of the following
		system of equations in
		\(\mathbb{Z}\):
		\begin{align*}
			& \sum_{i = 1}^n \sum_{j = 1}^k b_{ij1} Y_{ij} + c_i = 0 \\
			& \qquad \qquad \qquad \vdots \\
			& \sum_{i = 1}^n \sum_{j = 1}^k b_{ijk} Y_{ij} + c_i = 0.
		\end{align*}
		We can conclude that the lemma holds for single twisted equations in
		\(\mathbb{Z}^k\). It follows that the solution set to a system of \(m\)
		twisted equations in \(\mathbb{Z}^k\) will be constructible as stated in the
		lemma, from the solution set to a system of \(m\) of the above systems;
		that is a system of \(km\) equations in \(\mathbb{Z}\).
	\end{proof}

	Before we can prove Lemma \ref{free_abelian_twisted_n_reg_lem}, we need a
	slightly altered version of modular arithmetic, where we replace \(0\) with
	the quotient.

	\begin{ntn}
		For each \(n, \ r \in \mathbb{Z}_{\geq 0}\) with \(r > 0\), define
		\[
			n \modns^+ r = \left \{
			\begin{array}{cl}
				n \modns r & n \modns r \neq 0 \\
				r & n \modns r = 0.
			\end{array}
			\right.
		\]
	\end{ntn}

	We are now in a position to prove that multivariable solution languages to
	twisted equations in free abelian groups are accepted by multivariable finite
	state automata. We do this by expressing our equation as an identity of
	matrices, where the coefficients of the matrix determine the equation. This
	allows us to use the bound from Lemma \ref{eqn_bound_K_lem} to construct our
	automaton.

  \begin{lem}
    \label{free_abelian_twisted_n_reg_lem}
   	The multivariable solution language to a system of twisted equations in a
   	free abelian group, with respect to the standard generating set and normal
   	form, is accepted by a multivariable finite-state automaton.
  \end{lem}

  \begin{proof}
		Let \(\mathcal E\) be a system of $m$ twisted equations in \(\mathbb{Z}^k\)
		in \(n\) variables. Let \(\{a_1, \ \ldots, \ a_k\}\) denote the standard
		generating set for \(\mathbb{Z}^k\). By Lemma \ref{lem:Zsols}, there is a
		system of \(km\) equations \(\mathcal{F}\) in \(\mathbb{Z}\) with \(kn\)
		variables, such that the solution language to \(\mathcal{E}\) is equal to
		\[
			\mathcal{S}_\mathcal{E} = \left\{\left(a_1^{t_1} \cdots a_k^{t_k} , \ a_1^{t_{k + 1}}
			\cdots a_k^{t_{2k}}, \ \ldots , \ a_1^{t_{k(n - 1) + 1}} \cdots a_k^{t_{kn}} \right)
			\,\middle\vert\, (t_1, \ \ldots, \ t_{kn}) \text{ is a solution to } \mathcal{F}\right\}.
		\]

		We represent this new system \(\mathcal{F}\) via
		the identity $\mathbf{X}B=\mathbf{c}$ where
		\begin{itemize}
			\item $\mathbf{X}=(X_1,\ \ldots, \ X_{kn})$ is a row vector of $kn$ variables,
			\item $B = [b_{ij}]$ is a $kn\times km$ matrix of coefficients, and
			\item $\mathbf{c}\in\Z^{km}$ is a row vector of constants.
		\end{itemize}
		The constant of Lemma \ref{eqn_bound_K_lem} is then $K=\max(|c_1|, \ \ldots, \ |c_{km}|) + (kn)^{3/2}b_{\max}$.

		We can now show that the multivariable solution language is accepted by a
		\(kn\)-variable finite-state automaton, using the method described in Example
		\ref{exp_lin_n_reg_ex}. We define our automaton
		\(\mathcal{A}\) to have the set of states
		\[
			\{q_\mathbf{x} \mid \mathbf{x}=(x_1,\ \ldots \ ,x_{kn}) \in \mathbb{Z}^{kn}, \ |x_i| \leq
	   	K\},
		\]
		Our start state will be \(q_\mathbf{0}\), and \(q_{\mathbf{c}}\) will be our
		only accept state. Let \(\mathbf{w}_i\) be the \(kn\)-variable word with
		\(a_{i \modns^+ k}\) in the \(i\)th position, and \(\varepsilon\) elsewhere.
		We have an edge from \(q_\mathbf{x}\) to \(q_\mathbf{y}\) labelled with
		\(\mathbf{w}_j\) for all \(j\) such that \(\mathbf{x} + (b_{j1}, \ \ldots, \
		b_{j(kn)})  = \mathbf{y}\). By construction, the language accepted by
		\(\mathcal{A}\) is contained within \(\mathcal{S}_\mathcal{E}\). On the other hand, any word in \(\mathcal{S}_\mathcal{E}\) is accepted by \(\mathcal{A}\), by following an appropriate sequence as given by Lemma \ref{eqn_bound_K_lem}.
  \end{proof}

	We now consider the space complexity that is needed to construct the
	multivariable finite-state automaton defined in the proof of Lemma
	\ref{free_abelian_twisted_n_reg_lem}.

	\begin{rmk}
		\label{free_abelian_input_rmk}
		Before we can attempt to show anything about space complexity, we need to
		define the length of our input. Often, when talking about equations
		\(w \in F_V \ast G\), for a group \(G\) and set of variables \(V\), we take
		the length of \(w\) to be \(|w|\), with respect to the word metric on
		\(F_V \ast G\), which is inherited from \(F_V\) with respect to \(V\), and
		\(G\) with respect to a specified generating set.

		With free abelian, and also virtually abelian groups, we can write our
		equations more efficiently. Recall that we use \(a_1, \ \ldots, \ a_k\) to
		be the standard generating set of \(\mathbb{Z}^k\), when using
		multiplicative notation. By reordering an equation in \(n\) variables in
		\(\mathbb{Z}^k\), and we can assume it is in the form
		\[
			X_1^{b_1} \cdots X_n^{b_n} a_1^{c_1} \cdots a_k^{c_k} = 1,
		\]
		where \(X_1, \ \ldots, \ X_n \in V\), and
		\(b_1, \ \ldots, \ b_n \in \mathbb{Z}\). The length of this equation with
		respect to the word metric is
		\[
			\sum_{i = 1}^n |b_i| + \sum_{j = 1}^k |c_j|.
		\]
		However, since an integer \(r \in \mathbb{Z}\) can be stored using \(\log|r|
		+ C\) bits, for some constant \(C\), and we only need to store \(b_1, \
		\ldots, \ b_n, \ c_1, \ \ldots, \ c_k\), we can write this equation using
		\[
			\sum_{i = 1}^n \log|b_i| + \sum_{j = 1}^k \log|c_j| + Ckn
		\]
		bits. To write a twisted equation, we can first rearrange it to the form
		\[
			(X_1 B_1) \cdots (X_n B_n) a_1^{c_1} \cdots a_k^{c_k} = 1,
		\]
		for the matrices \(B_r = [b_{rij}]\) as described in the proof of Lemma
		\ref{lem:Zsols}. The matrices are all \(k \times k\) matrices, and therefore
		\(B_r\) can be stored using
		\[
			\sum_{i,j = 1}^k \log|b_{rij}| + C'k^2
		\]
		bits, for some constant \(C'\). %We will use this twisted equation format to store equations in
		%virtually abelian groups as well, however we will first need to prove
		%Proposition \ref{virtually_twisted_EP_n_reg_prop} in order to define that
		%format.
	\end{rmk}

	\begin{lem}
		\label{free_abelian_space_comp_lem}
		The multivariable finite-state automaton defined in Lemma
		\ref{free_abelian_twisted_n_reg_lem} can be constructed in non-deterministic
		quadratic space.
	\end{lem}

	\begin{proof}
		Let \(k\) be the rank of the free abelian group, \(\mathcal{E}\) be the
		system of equations, \(n\) be the number of variables, and \(m\) be the
		number of equations. We start by converting \(\mathcal{E}\) into the form
		\(\mathbf{X} B = \mathbf{c}\) (all we need to store is \(B\) and
		\(\mathbf{c}\)). Let \(I \in \mathbb{Z}_{\geq 0}\) be the length of the
		input.

		Index the equations \(w_1, \ \ldots, \ w_m\). We copy each equation in
		\(\mathcal{E}\) into the work tape, so our work tape will now have the same
		size as our input. We have assumed our equations are already in the form
		stated in Remark \ref{free_abelian_input_rmk}, and converting them to
		additive notation means they will be in the form
		\[
			\mathbf{Y}_1 B_1 + \cdots \mathbf{Y}_n B_n = \mathbf{d},
		\]
		where each \(B_i\) is a \(k \times k\) matrix, each \(\mathbf{Y}_i\) is a
		variable, and \(\mathbf{d} \in \mathbb{Z}^k\). We will now construct the
		matrix \(B\) and the vector \(\mathbf{c}\). We write \(\mathbf{Y}_i =
		(Y_{i1}, \ \ldots, \ Y_{ik})\), and \(B_1 = [b_{1ij}], \ \ldots, \ B_n =
		[b_{nij}]\). For each equation
		\[
			\mathbf{Y}_1 B_1 + \cdots \mathbf{Y}_n B_n = \mathbf{d},
		\]
		add the following vectors as columns to \(B\), and store them in the
		work tape:
		\[
			(b_{111}, \ \ldots, \ b_{nk1}), \ \ldots (b_{11k}, \ \ldots, \ b_{nkk}).
		\]
		The matrix \(B\) will at this point be a \(kn \times km\)
		matrix.
		%in the proof of Lemma \ref{lem:Zsols} we added zero rows or zero
		%columns to represent redundant equations in order to assume \(m = n\),
		%however we won't do that here. We will simply store \(m\) and \(n\), and
		%`simulate' the zero rows or columns when we need %them. 
		For each equation,
		we also append the entries of \(\mathbf{d}\) to the vector \(\mathbf{c}\).

		We now construct the states. Since our set of states is the set of all
		\(q_\mathbf{x}\) such that \(\mathbf{x}\in\mathbb{Z}^{kn}\) with
		each coordinate having absolute value at most \(K\), where \(K\) is from
		Lemma \ref{eqn_bound_K_lem}, we can construct the set of states by
		remembering the last state constructed, together with the bound \(K\), and
		proceeding in any `sensible' systematic manner, such as starting in one
		`corner', and running down each line in the `cube'. To do this, we need a
		memory that can store a vector of length \(kn\) at any time,
		with entries within \([-K, \ K]\).

		As in the proof of Lemma \ref{free_abelian_twisted_n_reg_lem}, \(K = \max(|c_1|,\ \ldots,\ |c_{km}|) +
		(kn)^{3/2}\cdot b_{\max}\), where \(b_{\max} = \max_{i, j} |b_{ij}|\). Recall
		\(I = \sum_i \log|c_i| + \sum_{i, j} \log|b_{ij}| + Ckn\), where \(C\)
		is a constant, as mentioned in Remark \ref{free_abelian_input_rmk}. Then
		\[
			\log K \leq \log|c_1| + \cdots + \log |c_{km}|  + \frac{3}{2} \log (kn)
			+ \log|b_{\max}| \leq \frac{3}{2} I
		\]
		So storing an integer within \([-K, \ K]\) requires \(\frac{3}{2} I\)
		bits, ignoring constants. Since \(kn \leq I\), storing a vector of length \(kn\)
		with entries in \([-K, \ K]\) requires at most \(\frac{3}{2}I^2\)
		bits.

		We can simply assign \(\mathbf{0}\) and \(\mathbf{c}\) as the start and
		accept states.

		We now need to compute the edges. Recall that we have an edge from
		\(q_\mathbf{x}\) to \(q_\mathbf{y}\) labelled with \(\mathbf{w}_j\) for all
		\(j\) such that \(\mathbf{x} + (b_{j1}, \ \ldots, \ b_{j(kn)})  =
		\mathbf{y}\), where \(\mathbf{w}_j\) is the \(kn\)-variable
		word with \(a_{i \modns^+ k}\) in the \(i\)th position and \(\varepsilon\)
		elsewhere. Therefore, we can go through the states systematically and
		add all of the outgoing edges, and we only need to remember the state we are on
		in order to compute and output its outgoing edges and their labels. To do this,
		we only need to record a vector of length \(kn\), the entries
		of which will lie in \([-K, \ K]\). As discussed before, this requires at
		most \(\frac{3}{2}I^2\) bits to store.
	\end{proof}

	In the next lemma, we show that the solution set to a system of equations in
	an arbitrary group can be expressed in terms of the solution set to a system
	of \emph{twisted} equations in a finite-index subgroup.

	\begin{lem}
		\label{lem:virtual_sols}
		Let \(G\) be a group, and \(T\) be a finite transversal of a normal subgroup
		\(H\) of finite index. Let \(\mathcal{S}\) be the solution set to a finite
		system of equations with \(n\) variables in \(G\). Then there is a finite
		set \(B \subseteq T^n\), and for each \(\mathbf{t} \in B\), there is a
		solution set \(A_\mathbf{t}\) to a system of twisted equations in \(H\),
		such that
		\[
			\mathcal{S} = \bigcup_{(t_1, \ \ldots, \ t_n) \in B}
			\left\{(h_1 t_1, \ \ldots, \  h_n t_n) \mid
			(h_1, \ \ldots, \ h_n) \in A_{(t_1, \ \ldots, \ t_n)} \right\}.
		\]
	\end{lem}

	\begin{proof}
	Let
		\begin{align}
			\label{virt_eqn}
			X_{i_{1j}}^{\epsilon_{1j}} g_{1j} \cdots X_{i_{pj}}^{\epsilon_{pj}} g_{pj}
			= 1
		\end{align}
		be a system of equations in \(G\), with variables \(X_1, \ \ldots, \ X_n\),
		where \(j \in \{1, \ \ldots, \ k\}\) for some constant \(k\), and
		\(\epsilon_{ij} \in \{-1, \ 1\}\) for all \(i\) and \(j\). Note that we can
		fold the constants that can occur before the first occurence of a variable
		\(X_{i_{1j}}\) in each equation into \(g_{pj}\) by conjugating. For each
		\(X_i\), define new variables \(Y_i\) and \(Z_i\) over \(H\) and \(T\)
		respectively such that \(X_i = Y_i Z_i\). We have that \(g_i = h_i t_i\) for
		some \(h_i \in H\) and \(t_i \in T\). Replacing these in (\ref{virt_eqn})
		gives
		\begin{align}
			\label{virt_subbed_in_eqn}
			(Y_{i_{1j}} Z_{i_{1j}})^{\varepsilon_{i_{1j}}} h_{1j} t_{1j} \cdots
			(Y_{i_{pj}} Z_{i_{pj}})^{\varepsilon_{i_{pj}}} h_{pj} t_{pj} = 1.
		\end{align}
		For each \(g \in G\), define \(\psi_g \colon G \to G\) by \(h \psi_g =  g h
		g^{-1}\). We will abuse notation, and extend this notation to define
		\(\psi_{Z_1}, \ \ldots, \ \psi_{Z_n}\). For each \(i\) and \(j\), let
		\[
			\delta_{ij} = \left\{
				\begin{array}{cl}
					0 & \epsilon_{ij} = 1 \\
					1 & \epsilon_{ij} = -1.
				\end{array}
			\right.
		\]
		It follows that (\ref{virt_subbed_in_eqn}) is equivalent to
		\begin{align*}
			(Y_{i_{1j}}^{\epsilon_{1j}} \psi_{Z_{i_{1j}}}^{\delta_{1j}}) Z_{i_{1j}}^{\epsilon_{1j}}
			h_{1j} t_{1j} \cdots
			(Y_{i_{pj}}^{\epsilon_{pj}} \psi_{Z_{i_{pj}}}^{\delta_{pj}}) Z_{i_{pj}}^{\epsilon_{pj}}
			h_{pj} t_{pj} = 1.
		\end{align*}
		By pushing all \(Y_i\)s and \(h_i\)s to the left, we obtain
		\begin{align}
			\label{virt_H_to_left_eqn}
			& (Y_{i_{1j}}^{\epsilon_{1j}} \psi_{Z_{i_{1j}}}^{\delta_{1j}})
			(h_{1j} \psi_{Z_{i_{1j}}}^{\epsilon_{1j}}) \cdots
			(Y_{i_{pj}}^{\epsilon_{pj}} \psi_{Z_{i_{pj}}}^{\delta_{pj}} \psi_{t_{(p - 1)j}}
			\psi_{Z_{i_{(p - 1)j}}}^{\epsilon_{(p - 1)j}} \cdots \psi_{t_{1j}}
			\psi_{Z_{i_{1j}}}^{\epsilon_{1j}})
			(h_{pj} \psi_{Z_{i_{pj}}}^{\epsilon_{pj}} \psi_{t_{(p - 1)j}} \cdots \psi_{t_{1j}}
			\psi_{Z_{i_{1j}}}^{\epsilon_{1j}}) \\
			\nonumber
			& Z_{i_{1j}}^{\epsilon_{1j}} t_{1j} \cdots Z_{i_{pj}}^{\epsilon_{pj}} t_{pj} = 1.
		\end{align}
		Note that if \((u_1, \ \ldots, \ u_n) \in T^n\) is a solution to the
		\(Z_i\)s within a solution to (\ref{virt_H_to_left_eqn}), then
		\(u_{i_1}^{\epsilon_1} t_1 \cdots u_{i_p}^{\epsilon_p} t_p \in H\). Let
		\(A\) be the set of all such tuples. Note that as \(T\) is finite, so is
		\(A\). Plugging a fixed \((u_1, \ \ldots, \ u_n) \in A\) into
		(\ref{virt_H_to_left_eqn}) gives a twisted equation in \(H\):
		\begin{align*}
			& (Y_{i_{1j}}^{\epsilon_{1j}} \psi_{u_{i_{1j}}}^{\delta_{1j}})
			(h_{1j} \psi_{u_{i_{1j}}}^{\epsilon_{1j}}) \cdots
			(Y_{i_{pj}}^{\epsilon_{pj}} \psi_{u_{i_{pj}}}^{\delta_{pj}} \psi_{t_{(p - 1)j}}
			\psi_{u_{i_{(p - 1)j}}}^{\epsilon_{(p - 1)j}} \cdots \psi_{t_{1j}}
			\psi_{u_{i_{1j}}}^{\epsilon_{1j}})
			(h_{pj} \psi_{u_{i_{pj}}}^{\epsilon_{pj}} \psi_{t_{(p - 1)j}} \cdots \psi_{t_{1j}}
			\psi_{u_{i_{1j}}}^{\epsilon_{1j}}) \\
			& u_{i_{1j}}^{\epsilon_{1j}} t_{1j} \cdots u_{i_{pj}}^{\epsilon_{pj}} t_{pj} = 1.
		\end{align*}
		Let \(B_{(u_1, \ \ldots, \ u_n)}\) be the solution set to the above system.
		It follows that the solution set to (\ref{virt_eqn}) equals
		\[
			\bigcup_{(u_1, \ \ldots, \ u_n) \in A} \{(f_1 u_1, \ \ldots, \ f_n u_n)
			\mid (f_1, \ \ldots, \ f_n) \in B_{(u_1, \ \ldots, \ u_n)}\}.
		\]
	\end{proof}

	The following proposition reflects a well-known fact about decidability of
	systems of equations in groups: if a group \(G\) has a finite index normal
	subgroup \(H\), such that there is an algorithm that determines if any
	system of twisted equations in \(H\) admits a solution, then there is an algorithm that determines if any system of (untwisted) equations \(G\) admits a solution.
	This fact turns out to be true regarding EDT0L solutions, and a variant of it
	is used in \cite{VF_eqns}.

	\begin{proposition}
		\label{virtually_twisted_EP_n_reg_prop}
		Let \(G\) be a group with a finite index normal subgroup \(H\), such that
		the multivariable solution language to any system of twisted equations in
		\(H\) is accepted by an \(n\)-variable finite-state automaton, for some \(n
		\in \mathbb{Z}_{>0}\), with respect to a generating set \(\Sigma\), and
		normal form \(\eta\).

		Then the multivariable solution language to any system of equations in \(G\)
		is accepted by an \(n\)-variable finite-state automaton, for some \(n \in
		\mathbb{Z}_{>0}\), with respect to the generating set \(\Sigma \cup T\), for any
		right transversal \(T\) of \(H\), and the normal form \(\zeta\), where
		\(g \zeta = (h \eta) t\), where \(h \in H\) and \(t \in T\) are (unique)
		such that \(g = ht\).
	\end{proposition}

	\begin{proof}
		We have from Lemma \ref{lem:virtual_sols}, that the solution language is
		a finite union across valid choices of transversal vectors \((t_1, \ \ldots,
		\ t_n)\) of
		\begin{align}
			\label{soln_lang_eqn}
			\{((h_1 \eta) t_1, \ \ldots, \ (h_n \eta) t_n) \mid (h_1, \ \ldots, \ h_n)
			\in A_{(t_1, \ \ldots, \ t_n)}\},
		\end{align}
		where \(A_{(t_1, \ \ldots, \ t_n)}\) is the solution set to a system of
		twisted equations in \(H\). Since the class of languages accepted by
		\(n\)-variable finite-state automata is closed under finite unions, it
		suffices to show that (\ref{soln_lang_eqn}) is accepted by an \(n\)-variable
		finite-state automaton.

		By our assumptions on \(H\), the language
		\[
			\{(h_1 \eta, \ \ldots, \ h_n \eta) \mid (h_1, \ \ldots, \ h_n) \in
			A_{(t_1, \ \ldots, \ t_n)}\}
		\]
		% is accepted by an \(n\)-variable finite-state automaton.
		% By Lemma \ref{EDT0L_diff_hash_lem},
		% \begin{align}
		% 	\label{soln_lang_2_eqn}
		% 	\{h_1 \#_1 \cdots \#_{n - 1} h_n \#_n \mid (h_1, \ \ldots, \ h_n) \in
		% 	A_{(t_1, \ \ldots, \ t_n)}\}
		% \end{align}
		% is EDT0L. Let \(\phi\) be the free monoid homomorphism defined by
		% \begin{align*}
		% 	\#_i \phi = \left\{
		% 	\begin{array}{cl}
		% 		t_i \# & i \in \{1, \ \ldots, \ n - 1\} \\
		% 		t_n & i = n.
		% 	\end{array}
		% 	\right.
		% \end{align*}
		% Applying \(\phi\) to (\ref{soln_lang_2_eqn}) will give
		% (\ref{soln_lang_eqn}), and it follows from Lemma
		% \ref{EDT0L_closure_properties_lem} (5) that (\ref{soln_lang_eqn}) is EDT0L.
		%
		% To prove the second statement, note that the language
		% \[
		% 	\{(h_1, \ \ldots, \ h_n) \mid (h_1, \ \ldots, \ h_n) \in A_{(t_1, \
		% 	\ldots, \ t_n)}\}
		% \]
		is accepted by an \(n\)-variable finite-state automaton \(\mathcal{M}\), for
		any valid choice of transversal vector \((t_1, \ \ldots, \ t_n)\). We can
		therefore modify this automaton to accept
		\[
			\{(h_1 t_1, \ \ldots, \ h_n t_n) \mid (h_1, \ \ldots, \ h_n) \in A_{(t_1,
			\ \ldots, \ t_n)}\}
		\]
		We do this by adding a new state \(q\), and with an edge labelled \((t_1, \
		\ldots, \ t_n)\) from every accept state of \(\mathcal{M}\), and making \(q\)
		the only accept state. By construction, this accepts the stated language.
	\end{proof}

	\begin{rmk}
		Before we can talk about the space complexity of equations in virtually
		abelian groups, we need to define the size of our input. We do this by using
		the finite-index normal free abelian subgroup, as these have an efficient
		way of storing equations.

		Let \(G\) be a group with a finite-index subgroup \(H\), and suppose that
		there is a convention for input sizes of twisted equations in \(H\) (see
		Remark \ref{free_abelian_input_rmk} for free abelian groups). We define the
		length of a system of equations in \(G\) to be the length of the system of
		twisted equations in \(H\) derived in Lemma \ref{lem:virtual_sols}.
	\end{rmk}

	\begin{lem}
		\label{fin_ind_space_comp_lem}
		Let \(G\), \(H\), \(n\) and \(\Sigma\) be defined as in Proposition
		\ref{virtually_twisted_EP_n_reg_prop}, and suppose the multivariable
		finite-state automaton that accepts a system of twisted equations in \(H\),
		in the statement of Proposition \ref{virtually_twisted_EP_n_reg_prop}, is
		constructible in \(\mathsf{NSPACE}(f)\), where \(f \colon \mathbb{Z}_{\geq
		0} \to \mathbb{Z}_{\geq 0}\) is a function. Then the automaton that accepts
		a system of equations in \(G\) is also constructible in
		\(\mathsf{NSPACE}(f)\).
	\end{lem}

	\begin{proof}
	 	By Lemma \ref{finite_union_n_reg_space_complexity_lem}, it suffices to show
		that each automaton that accepts a language
		\[
			\{(h_1 t_1, \ \ldots, \ h_n t_n) \mid (h_1, \ \ldots, \ h_n) \in A_{(t_1,
			\ \ldots, \ t_n)}\},
		\]
		where \(A_{(t_1, \ldots, t_n)}\) is as defined in the proof of Proposition
		\ref{virtually_twisted_EP_n_reg_prop}. Recall that this is constructed from
		the automaton \(\mathcal{M}\) that accepts a system of twisted equations in
		\(H\) by adding one additional state \(q\), and edges from each accept state
		to \(q\), all labelled \((t_1, \ \ldots, \ t_n)\), and then by making \(q\)
		the only accept state. We do this by modifying the algorithm that constructs
		\(\mathcal{M}\) to add the state \(q\) at the beginning, then perform the
		algorithm that constructs \(\mathcal{M}\), except whenever we would label a
		state \(p\) as an accept state, we instead add an edge from \(p\) to \(q\),
		labelled by \((t_1, \ \ldots, \ t_n)\). This does not use a longer work tape
		than the algorithm that constructs \(\mathcal{M}\).
	\end{proof}

	We now have enough to prove our first main result. Our generating set is the
	union of the standard generating set \(\Sigma\) of the finite index free
	abelian subgroup together with a right transversal \(T\). We use the normal
	form
	\[
		\{at \mid a \in \Sigma, \ t \in T\}.
	\]

  \begin{theorem}
		\label{theorem:VA_n_reg}
    Multivariable solution languages to systems of equations in virtually
    abelian groups with \(n\) variables are accepted by \(n\)-variable
    finite-state automata.
  \end{theorem}

  \begin{proof}
    This follows from Proposition \ref{virtually_twisted_EP_n_reg_prop}, together
		with the fact that multivariable solution languages to systems of twisted
		equations in free abelian groups are accepted by \(n\)-variable
		finite-state automata (Lemma \ref{free_abelian_twisted_n_reg_lem}).
  \end{proof}

	Lemma \ref{n_reg_implies_EDT0L_lem} now gives us the following result.

	\begin{cor}
		\label{VA_EDT0L_cor}
		The \(\#\)-joined solution languages to systems of equations in virtually
		abelian groups are EDT0L.
	\end{cor}

	\begin{rmk}
		Corollary \ref{VA_EDT0L_cor} uses the normal form defined by writing an
		element of a virtually abelian group as a product of a word in the
		finite-index free abelian normal subgroup, written in standard normal form,
		with an element of the (finite) transversal for that subgroup.

		We can change our generating set to any other generating set, and there will
		exist a normal form such that solution languages are still EDT0L. Adding a
		new generator does not change the language at all, as we can keep the normal
		form the same, and so our new generator will not appear in any normal form
		word. To remove a redundant generator \(c\), we can fix a word \(w_c\) over
		the remaining generators and their inverses that represents the same element
		as \(c\), and apply the free monoid homomorphism that maps \(c\) to \(w_c\).
		This corresponds to changing the normal form used by replacing every
		occurence of \(c\) with \(w_c\).

		Changing the normal form is more difficult. In \cite{eqns_hyp_grps},
		Section 5, Ciobanu and Elder show that changing between quasigeodesic
		normal forms will not affect whether or not the solution language to a given
		system is EDT0L. This relies on the fact that in a hyperbolic group \(G\),
		the set of all pairs \((u, \ v)\) of \((\lambda, \ \mu)\)-quasigeodesics
		such that \(u =_G v\) is accepted by a 2-variable finite state automaton.
		Unfortunately, this doesn't work in \(\mathbb{Z}^2\), so a different approach would be required to preserve the EDT0L status of the
		language when changing between normal forms
		in virtually abelian groups.
	\end{rmk}

	We now combine the various lemmas on the space complexity of the algorithms we
	have used to construct multivariable finite-state automata and EDT0L systems
	to show the following.

	\begin{proposition}\label{prop:quadspace}
		The multivariable finite-state automaton from Theorem \ref{theorem:VA_n_reg}
		and the EDT0L system from Corollary \ref{VA_EDT0L_cor} can be constructed
		in non-deterministic quadratic space.
	\end{proposition}

	\begin{proof}
		The fact that the multivariable finite-state automaton can be constructed in
		non-deterministic quadratic space follows from Lemma
		\ref{fin_ind_space_comp_lem} and Lemma \ref{free_abelian_space_comp_lem}.
		Using this fact with Lemma \ref{n_reg_implies_EDT0L_lem} gives the
		second statement.
	\end{proof}

	To understand the the growth of the $\#$-joined solution language, we need the following Lemma.
	\begin{lem}\label{lem:twistedpoly}
		Let $A\subseteq(\Z^k)^n$ be the solution set to a system of twisted equations in $\Z^k$ (with $n$ variables). Then $A$ is a polyhedral subset of $\Z^{kn}$.
	\end{lem}
	\begin{proof}
		By Lemma \ref{lem:Zsols}, $A$ may be viewed as the set of solutions to a system of (non-twisted) equations in $\Z$, with $kn$ variables, with each element of $A$ given as a vector in $\Z^{kn}$, with respect to the standard basis of $\Z^{kn}$. Now a single such equation in $\Z$ may be expressed as
		\begin{equation*}
			\sum_{i=1}^{kn} a_ix_i = b
		\end{equation*}
		for variables $x_i$ and constants $a_i, \ b\in\Z$. Therefore the solution set to such an equation has the form
		\begin{align*}
			\left\{(x_1, \ \ldots, \ x_{kn})\in\Z^{kn}\, \ \middle\vert\, \sum_{i=1}^{kn} a_ix_i = b\right\} = \left\{\mathbf{x}\in\Z^{kn} \, \middle\vert \, \mathbf{a}\cdot\mathbf{x}=b\right\}
		\end{align*}
		and is thus an elementary set (see Definition \ref{def:polyhedral}). The solution set to a system of equations is then the intersection of finitely many elementary sets, and is therefore a polyhedral set by the definition.

	\end{proof}

	We can now use the polyhedral structure of solution sets in $\Z^k$ to prove the following Proposition about the growth of solution languages in virtually abelian groups.
	\begin{proposition}\label{prop:EDT0Lrational}
		The \(\#\)-joined solution language of any system of equations in a virtually abelian group has
		rational growth series.
	\end{proposition}

	\begin{proof}
		As before, let $G$ be a virtually abelian group and let $\Z^k$ denote a free abelian normal subgroup of finite index, and $T$ a choice of transversal. The normal form on $\Z^k$ given by the standard basis vectors is denoted $\eta$. By Lemma \ref{lem:virtual_sols}, the solution language is given by a finite union of sets of the form
		\begin{equation}\label{eq:solsgrowth}
		\{(h_1\eta)t_1\# (h_2\eta)t_2\#\cdots\# (h_n\eta)t_n\mid (h_1, \ \ldots, \ h_n)\in A_{\mathbf{t}}\}
		\end{equation}
		where $n$ is the number of variables, $\mathbf{t}=(t_1, \ \ldots, \ t_n)$ is some subset of $T^n$, and each $A_{\mathbf t}$ is the solution set to some system of twisted equations in $\Z^k$.

		Now, the word $(h_1\eta)t_1\#\cdots\#(h_n\eta)t_n\in \left(T\cup\{\#\}\cup\{\pm e_i\mid 1\leq i\leq kn\}\right)^*$ has length $2n-1+|(h_1, \ \ldots, \ h_n)|$. So the growth series of the set \eqref{eq:solsgrowth} is equal to the growth series of $A_{\mathbf t}$ multiplied by $z^{2n-1}$. That is,
		\begin{equation*}
			z^{2n-1}\sum_{m=0}^\infty \#\{(h_1, \ \ldots, \ h_n)\in A_{\mathbf t} \mid
			|(h_1, \ \ldots, \ h_t)|=m\} z^m.
		\end{equation*}

		Since each $A_{\mathbf{t}}$ is polyhedral by Lemma \ref{lem:twistedpoly}, Corollary \ref{cor:rationalpoly} implies that their growth series (with the weight of each generator equal to $1$ in this case) is rational, and hence the growth series of \eqref{eq:solsgrowth} is also rational. So the growth series of the solution language is a finite sum of rational functions, and therefore rational itself.
	\end{proof}

\begin{rmk}
	We note that the language above will not be context-free in general. For
	example, suppose the underlying group is $\Z=\langle x\rangle$, and consider
	the equation $X=Y=Z$ (more formally the system of equations
	$XY^{-1}=YZ^{-1}=1$). In the notation of this paper, the set of solutions is
	$\{a^m \# a^m \# a^m \mid m \in\Z\}$, which is not context-free over the alphabet
	\(\{a, \ a^{-1}, \ \#\}\) by standard techniques.

	Thus we have a large class of EDT0L languages, with rational growth series, which are not, in general, context-free.
\end{rmk}

%%%%%%%%%%%%%%%%%%%%%%%%%%%%%%%%%%%%%%%%%%%%%%%%%%%%%%%%%%%%%%%%%%%%%%%%%%%%%%%%%%%%%%%%%%%%%%%%%%%%%%%%%%%%%%%%%%%%%%%%%%%%%%%%%%%%%%%%%%%%%%%%%%%%%%%%%%%%%%%%%%%%%%%
%%%%%%%%%%%%%%%%%%%%%%%%%%%%%%%%%%%%%%%%%%%%%%%%%%%%%%%%%%%%%%%%%%%%%%%%%%%%%%%%%%%%%%%%%%%%%%%%%%%%%%%%%%%%%%%%%%%%%%%%%%%%%%%%%%%%%%%%%%%%%%%%%%%%%%%%%%%%%%%%%%%%%%%

\section{Relative growth of algebraic sets}\label{sec:growth}
	We now study the nature of algebraic sets from a different point of view. Expanding on the theme of Proposition \ref{prop:EDT0Lrational}, we consider the \emph{growth} of algebraic sets, this time as sets of tuples of group elements, with respect to a natural metric inherited from the word metric on the group.

	The usual notion of the growth function of a group can be altered by restricting to a subset. This is known as \emph{relative growth}. The study of relative growth of subgroups in particular has attracted significant interest, for example Davis-Olshanskii \cite{DO}, and recently Cordes-Russell-Spriano-Zalloum \cite{Morse}. Here, we define and study the relative growth of algebraic sets. Since such a set is a subset of $G^n$, rather than $G$ itself, we must decide how to assign lengths to tuples. We do this in perhaps the most obvious way, by taking the sum of the lengths of the components (see Definition \ref{def:alggrowth}).

	Since the growth of virtually abelian groups is always polynomial (that is, the number of elements of length $n$ is at most polynomial in $n$), it is clear that the same will be true of algebraic sets. Instead, we study the growth series, the formal power series associated to the relative growth function of an algebraic set, and show that this is always a rational function (see Theorem \ref{thm:soln_sets_rational}). This means that there exists a set of unique geodesic representatives for each algebraic set, which has rational growth series as a language.

	An alternative approach which avoids the need to define the length of $n$-tuples of group elements is to study the \emph{multivariate} growth series, the formal power series in $n$ variables, which correspond to the $n$ variables of the system of equations in question (see Definition \ref{def:alggrowth}). In this case, we have the weaker result that the series is always holonomic (Corollary \ref{cor:holonomic}).

%	\subsection{Structure and growth of virtually abelian groups}

	From now on we will assume that $G$ is virtually abelian with a normal, finite index subgroup isomorphic to $\Z^k$ for some positive integer $k$.

	\begin{dfn}\label{def:growth}
		Let $G$ be generated by a finite set $S$ and suppose $S$ is equipped with a weight function $\|\cdot\|\colon S\to \Z_{>0}$. This naturally extends to $S^*$ so that $\|s_1s_2\cdots s_k\|=\sum_{i=1}^k\|s_i\|$.
		\begin{enumerate}
		\item Define the weight of a group element as
		\[\|g\|=\min\{\|w\|\mid w\in S^*, \ w=_Gg\}.\] Any word representing $g$ whose weight is equal to $\|g\|$ will be called \emph{geodesic}. This coincides with the usual notion of word length when the weight of each non-trivial generator is equal to $1$.
		\item Let $V\subseteq G$ be any subset. Then the \emph{relative weighted growth function} of $V$ relative to $G$, with respect to $S$, is defined as \[\sigma_{V\subseteq G,S}(m)=\#\{g\in V\mid \|g\|=m\}.\]
		For simplicity of notation, we will write $\sigma_V(m)$ when the other information is clear from context.
		\item The corresponding \emph{weighted growth series} is the formal power series \[\mbS_{V\subseteq G,S}(z)=\sum_{m=0}^\infty \sigma_{V\subseteq G,S}(m)z^m.\]
		\end{enumerate}
	\end{dfn}

	Benson proved in \cite{Benson} that the series $\mbS_{G\subseteq G}(z)$ is always rational (that is, the standard growth series of $G$), and the first named author proved in \cite{Evetts} that for any subgroup $H$ of $G$, the series $\mbS_{H\subseteq G}(z)$ is always rational. Both of these results hold regardless of the choice of finite weighted generating set. As discussed, we wish to apply these ideas to algebraic sets, which are subsets of $G^n$ in general, for some positive integer $n$. Therefore, we extend Definition \ref{def:growth} as follows.
	\begin{dfn}\label{def:alggrowth}
		Let $G$ be generated by a finite set $S$, equipped with a weight function $\|\cdot\|$.
		\begin{enumerate}
			\item 	Let $\mathbf{x}=(x_1, \ \ldots, \ x_n)\in G^n$ be any $n$-tuple of elements of $G$.  Define the weight of $\mathbf{x}$ as follows: \[\|\mathbf{x}\|=\min\left\{\sum_{i=1}^n \|v_i\|\, \ \middle\vert\, v_i\in S^*,~ v_i=_G x_i,~1\leq i\leq n\right\}=\sum_{i=1}^n\|x_i\|.\]
			\item Let $V\subseteq G^n$ be any set of $n$-tuples of elements. Then the \emph{relative weighted growth function} of $V$ is defined as the function \[\sigma_{V\subseteq G^n,S}(m)=\#\{\mathbf{x}\in V\mid \|\mathbf{x}\|=m\}.\]
			\item The corresponding (univariate) \emph{weighted growth series} is	\[\mbS_{V\subseteq G^n,S}(z)=\sum_{m=0}^\infty\sigma_{V\subseteq G^n,S}(m)z^m\in\Q[[z]].\]
			\item The \emph{multivariate growth series} is
			\[\mathbb{M}_{V\subseteq G^n,S}(z_1,\ldots,z_n)=\sum_{\mathbf{x}\in V\subseteq G^n} z_1^{\|x_1\|}\cdots z_n^{\|x_n\|}\in\Q[[z_1,z_2,\ldots,z_n]].\]
		\end{enumerate}
	We will suppress some or all of the subscripts when it is clear what the notation refers to.
	\end{dfn}

	With these definitions, we can state the main result of this section.
	\begin{theorem}
		\label{thm:soln_sets_rational}
		Let $G$ be a virtually abelian group. Then every algebraic set of $G$ has rational weighted growth series with respect to any finite generating set.
	\end{theorem}

	\subsection{Structure of virtually abelian groups}
	To prove the Theorem, we will extend the framework used in \cite{Benson} and \cite{Evetts} to apply to our setting. We give the necessary definitions and results below, and refer the reader to the above mentioned articles for full details.

	\begin{dfn}
		As above, fix a finite generating set $S$ for $G$.
		\begin{enumerate}
			\item We define $A=S\cap\Z^k$ and $B=S\setminus A$. Any word in $B^*$ will be called a \emph{pattern}.
			\item Let $A=\{x_1, \ \ldots, \ x_r\}$, and $\pi=y_1y_2\cdots y_l$ be some pattern (with each $y_i\in B$). Then a word in $S^*$ of the form
			\begin{equation}\label{eq:patternedword}
			w=x_1^{i_1}x_2^{i_2}\cdots x_r^{i_r} y_1 x_1^{i_{r+1}}x_2^{i_{r+2}}\cdots x_r^{i_{2r}} y_2 \cdots y_l x_1^{i_{lr+1}}x_2^{i_{lr+2}}\cdots x_r^{i_{lr+r}}
			\end{equation}
			for non-negative integers $i_j$ is called a \emph{$\pi$-patterned word}. For a fixed $\pi\in B^*$, denote the set of all such words by $W^\pi$.
		\end{enumerate}
	\end{dfn}
	This definition allows us to identify patterned words with vectors of non-negative integers, by focussing on just the powers of the generators in $A$ as follows.
	\begin{dfn}\label{def:phipi}
		Fix a pattern $\pi$ of length $l$, and write $m_\pi=lr+r$. Define a bijection $\phi_{\pi}\colon W^{\pi}\to\Z_{\geq0}^{m_\pi}$ via
		 \[\phi_\pi\colon x_1^{i_1}x_2^{i_2}\cdots x_r^{i_r} y_1 x_1^{i_{r+1}}x_2^{i_{r+2}}\cdots x_r^{i_{2r}} y_2 \cdots y_l x_1^{i_{lr+1}}x_2^{i_{lr+2}}\cdots x_r^{i_{lr+r}}\mapsto (i_1, \ i_2, \ \ldots, \ i_{lr+r}).\]
	\end{dfn}
	This bijection will allow us to count subsets of $\Z^{m_\pi}$ in place of sets of words. We apply the weight function $\|\cdot\|$ to $\Z^{m_\pi}$ in the natural way, weighting each coordinate with the weight of the corresponding $x\in A$. More formally, we have
	\begin{equation*}
	\|(i_1, \ \ldots, \ i_{m_\pi})\| := \sum_{j=1}^{m_\pi} i_j\|x_{j\modns^+ r}\|.
	\end{equation*}
	Then $\phi_\pi$ preserves the weight of words in $W^\pi$, up to a constant:
	\begin{equation*}
	\|w\phi_\pi\| = \|w\| - \|\pi\|.
	\end{equation*}
	Fix a transversal $T$ for the cosets of $\Z^k$ in $G$. Note that, since $\Z^k$ is a normal subgroup, we can move each $y_i$ in the word \eqref{eq:patternedword} to the right, modifying only the generators from $A$, and we have $\overline{w}\in\Z^k\overline{\pi}$. Thus $\overline{W^\pi}\subset\Z^kt_\pi$ for some $t_\pi\in T$ where $\overline{\pi}\in\Z^kt_\pi$.

	It turns out that we can pass from a word $w\in W^\pi$ to the normal form (with respect to $T$ and the standard basis for $\Z^k$) of the element $\overline{w}$ using an integral affine transformation.
	\begin{proposition}[Section 12 of \cite{Benson}]\label{prop:affA}
		For each pattern $\pi\in B^*$, there exists an integral affine transformation $\cA^\pi\colon\Z^{m_\pi}_{\geq0}\to\Z^k$ such that $\overline{w}=\left(w\phi_\pi\cA^\pi\right)t_\pi$ for each $w\in W^\pi$.
	\end{proposition}

	Observe that the union $\bigcup W^\pi$ of patterned sets taken over all patterns $\pi$ contains a geodesic representative for every group element (since any geodesic can be arranged into a patterned word without changing its image in the group). However, this is an infinite union, since patterns are simply elements of $B^*$.

	Consider the extended generating set $\widetilde{S}$
	defined as follows:
	\[
		\widetilde{S}=\{s_1s_2\cdots s_c\mid s_i\in S,~1\leq c\leq [G\colon\Z^k]\}.
	\]
	Define a weight function $\|\cdot\|_{\sim}\colon\widetilde{S}\to \mathbb{Z}_{> 0}$ via
	$\|s_1s_2\cdots s_c\|_\sim=\sum_{i=1}^c\|s_i\|$. Notice that although group elements will have different lengths with respect to this new generating set, we have $\|g\|_\sim=\|g\|$ for any $g\in
	G$. Thus the weighted growth functions, and hence series, of any subset $V\subseteq G$
	with respect to $S$ and $\widetilde{S}$ are equal. The following fact shows that
	passing to this extended generating set means we only need consider finitely
	many patterns.
	\begin{proposition}[11.3 of \cite{Benson}]
		Every element of $G$ has a geodesic representative with a pattern whose length (with respect to $\wt{S}$) does not exceed $[G\colon\Z^k]$.
	\end{proposition}
	\begin{dfn}
		Let $P$ denote the set of patterns of length at most $[G\colon\Z^k]$ (with respect to $\wt{S}$).
	\end{dfn}
	From now on we will implicitly work with the extended generating set, allowing us to restrict ourselves to the finite set of patterns $P$.

	We now reduce each $W^\pi$ so that we have only a single geodesic representative for each element of $G$.

	\begin{theorem}[Section 12 of \cite{Benson}]\label{thm:Upi}
		For each $\pi\in P$, there exists a set $U^\pi\subset W^\pi$ such that every word in $U^\pi$ is geodesic, every element in $G$ is represented by some word in $\bigcup_{\pi\in P} U^\pi$, and no two words in $\bigcup_{\pi\in P} U^\pi$ represent the same element. Furthermore, each $U^\pi\phi_\pi$ is a polyhedral set in $\Z^{m_\pi}$.
	\end{theorem}

	\begin{cor}
		The weighted growth series $\mbS_{G\subseteq G}(z)$ of $G$ is rational, with respect to all generating sets.
	\end{cor}
	\begin{proof}
		The growth series $\mbS_{G\subseteq G}$ is precisely the growth series of $\bigcup_{\pi\in P} U^\pi$ as a set. From Definition \ref{def:phipi} we have \[\mbS_{U^\pi\subseteq G}(z) = z^{\|\pi\|}\mbS_{U^\pi\phi_\pi}(z)\] and thus \[\mbS_{G\subseteq G}(z)=\sum_{\pi\in P}z^{\|\pi\|}\mbS_{U^\pi\phi_\pi}(z)\] is rational, since each $\mbS_{U^\pi\phi_\pi}(z)$ is a positive polyhedral set and hence rational by Proposition \ref{prop:polyhedralrationalgrowth}

	\end{proof}

\subsection{Univariate growth series of algebraic sets}

We can now demonstrate our main result. This will be a consequence of a more general rationality criterion. First, we make the following definitions, extending the framework explained above to $n$-tuples of group elements.
\begin{dfn}
	Let $\upi=(\pi_1, \ \ldots, \ \pi_n)\in P^n$ be a tuple of patterns, with respect to $\wt{S}$.
	\begin{enumerate}
		\item Let $W^{\upi}=W^{\pi_1}\times\cdots\times W^{\pi_n}$ and $U^{\upi}=U^{\pi_1}\times\cdots\times U^{\pi_n}\subset\left(S^*\right)^n$. Note that $U^{\upi}$ is a polyhedral set by Proposition \ref{prop:polyclosed}.
		\item Let $m_{\upi}=\sum_{i=1}^n m_{\pi_i}$, and $\|\upi\|=\sum_{i=1}^n\|\pi_i\|$.
		\item Define a map $\phi_{\upi}\colon W^{\upi}\to\Z_{\geq0}^{m_{\upi}}$ in the natural way via
		\[(w_1, \ \ldots, \ u_n)\mapsto (w_1\phi_{\pi_1}, \ \ldots, \ u_n\phi_{\pi_n}).\] As in the above discussion, $\phi_{\upi}$ preserves the weight of words, up to a constant, i.e.
		\[\|(w_1, \ \ldots, \ u_n)\phi_{\upi} \|=\sum_{i=1}^n\|w_i\| - \|\upi\|.\]
		\item Given $\cA^{\pi_i}$ as in Proposition \ref{prop:affA}, define an integral affine transformation $\cA^{\upi}\colon\Z_{\geq0}^{\upi}\to\Z^{kn}$ in the natural way via
		\[(x_1, \ \ldots, \ x_n)\mapsto\left(x_1\cA^{\pi_1}, \ \ldots, \ x_n\cA^{\pi_n} \right)\in\Z^k\times\cdots\times\Z^k.\]
	\end{enumerate}
\end{dfn}

Now we define a class of subsets of finitely generated virtually abelian groups which is particularly amenable to study using the tools we have described.
\begin{dfn}\label{def:CWP}
	Let $T$ be a choice of transversal for the finite index normal subgroup $\Z^k$. A subset $V\subseteq G^n$ will be called \emph{coset-wise polyhedral} if, for each $\mathbf{t}=(t_1, \ \ldots, \ t_n)\in T^n$, the set
	\[V_{\mathbf{t}} =\left\{\left(g_1t_1^{-1}, \ g_2t_2^{-1}, \ \ldots, \ g_nt_n^{-1}\right)\mid (g_1, \ \ldots, \ g_n)\in V,~g_i\in\Z^kt_i\right\}\subseteq\Z^{kn}\]
	is polyhedral.
\end{dfn}
\begin{rmk}
	Note that the definition is independent of the choice of $T$. Indeed, suppose that we chose a different transversal $T'$ so that for each $t_j\in T$ we have $t'_j\in T'$ with $\Z^kt_j=\Z^kt'_j$. Then there exists $y_j\in\Z^k$ with $t_j=y_jt'_j$ for each $j$, and so $g{t'_j}^{-1}=g_jt_j^{-1}y_i$ for any $g\in\Z^kt_j=\Z^k t'_j$. So changing the transversal changes the set $V_{\mathbf t}$ by adding a constant vector $(y_1, \ \ldots, \ y_n)$, and so it remains polyhedral by Proposition \ref{prop:polyaffine}.
\end{rmk}

%It is not hard to show that any subgroup of $G$ is coset-wise polyhedral. The following Theorem is therefore in some sense a generalisation of Theorem 3.3 of \cite{Evetts}, namely that every subgroup has rational relative growth series.

As an example of Definition \ref{def:CWP}, we provide a brief proof that subgroups are coset-wise polyhedral.
\begin{proposition}\label{prop:subgroupsCWP}
	Let $G$ be a virtually abelian group, with normal free abelian subgroup $\Z^k$, and let $H$ be any subgroup. Then $H$ is coset-wise polyhedral.
\end{proposition}
\begin{proof}
	By the Second Isomorphism Theorem, $H$ is itself virtually abelian, with finite-index (free) abelian subgroup $H\cap\Z^k$. Furthermore, $c:=[H\colon H\cap\Z^k]\leq[G\colon\Z^k]=:d$. Choose a set of representatives $\{t_1,\ \ldots,\ t_c\}$ for the cosets of $H\cap\Z^k$ in $H$, and extend this to a set of representatives $\{t_1,\ \ldots, \ t_c, t_{c+1}, \ \ldots, \ t_d\}$ for the cosets of $\Z^k$ in $G$. For each $t_i$ with $i\leq c$, the set
	\[H_{t_i}=\left\{ht_i^{-1}\,\middle\vert\, h\in H,~h\in\Z^kt_i\right\} = \left\{ht_i^{-1}\,\middle\vert\, h\in\left(H\cap\Z^k\right)t_i\right\}=H\cap\Z^k.\]
	For $i>c$, $H_{t_i}$ is empty. Now since $H\cap\Z^k$ is free abelian, it is a polyhedral set when viewed as a subset of $\Z^k$. The empty set is also polyhedral (as, say, the intersection of a pair of disjoint hyperplanes). Hence $H$ is coset-wise polyhedral.
\end{proof}
In light of Proposition \ref{prop:subgroupsCWP}, the following Theorem is in some sense a generalisation of Theorem 3.3 of \cite{Evetts}, namely that every subgroup has rational relative growth series.

\begin{theorem}\label{thm:rationaltuples}
	Let $G$ be virtually abelian, with normal free abelian subgroup $\Z^k$, and let $S$ be any finite weighted generating set. If $V\subseteq G^n$ is coset-wise polyhedral, then the weighted growth series $\mbS_{V\subseteq G^n,S}(z)$ is a rational function.
\end{theorem}
\begin{proof}
	Fix a transversal $T$. For each $\mathbf{t}\in T^n$, let $P_{\mathbf t}\subset P^n$ denote the set of $n$-tuples of patterns of the form $\upi=(\pi_1, \ \ldots, \ \pi_n)$ where each $\pi_i\in\Z^kt_i$. Let $U^{\mathbf \pi}=U^{\pi_1}\times\cdots\times U^{\pi_n}\subset\left(S^*\right)^n$. Then by Theorem \ref{thm:Upi}, the disjoint union $\bigcup_{\upi\in P_{\mathbf t}}U^{\upi}$ consists of exactly one $n$-tuple of geodesic representatives for each $n$-tuple in $\Z^kt_1\times\cdots\times\Z^kt_n$. We are only interested in $n$-tuples of elements which lie in the set $V$. Each element of $V$ lies in a unique product of cosets, so we partition $V$ into such products:
	\begin{align}\label{eq:Vunion}
	V = \bigcup_{\mathbf{t}\in T^n}\left\{(g_1, \ \ldots, \ g_n)\in V\mid g_i\in\Z^kt_i\right\} = \bigcup_{\mathbf{t}\in T^n}\left\{(g_1, \ \ldots, \ g_n)\in G^n\mid (g_1t_1^{-1}, \ \ldots, \ g_nt_n^{-1})\in V_{\mathbf{t}} \right\}.
	\end{align}
	Now, for a fixed $\mathbf{t}$, $(g_1, \ \ldots, \ g_n)$ has a unique geodesic representative in the set $U^{\upi}$, for some $\upi\in P_{\mathbf t}$ determined by $\mathbf{t}$. So the growth series of each component in the union \eqref{eq:Vunion} is equal to the growth series of the set
	\begin{equation*}
	\bigcup_{\upi\in P_{\mathbf t}}\left\{(u_1, \ \ldots, \ u_n)\in U^{\upi}\mid (u_1\phi_{\pi_1}\cA^{\pi_1}, \ \ldots, \ u_n\phi_{\pi_n}\cA^{\pi_n})\in V_{\mathbf t} \right\} = \bigcup_{\upi\in P_{\mathbf t}} V_{\mathbf t}(\phi_{\upi}\cA^{\upi})^{-1}\cap U^{\upi}.
	\end{equation*}

	Applying the map $\phi_{\upi}$ to a component of the union yields the set
	\begin{equation*}
	\left\{(u_1\phi_1, \ \ldots, \ u_n\phi_n)\in U^{\upi}\phi_{\upi}\mid (u_1\phi_{\pi_1}\cA^{\pi_1}, \ \ldots, \ u_n\phi_{\pi_n}\cA^{\pi_n})\in V_{\mathbf t}\right\} = V_{\mathbf t}\left(\cA^{\upi}\right)^{-1}\cap U^{\upi}\phi_{\upi}.
	\end{equation*}
	Now by Propositions \ref{prop:polyclosed} and \ref{prop:polyaffine}, this last set is polyhedral, and so has rational growth. Since both $T^n$ and $P_{\mathbf t}$ are finite, the growth series of $V$ is a finite sum of growth series of sets of the form $V_{\mathbf t}\left(\cA^{\upi}\right)^{-1}\cap U^{\upi}\phi_{\upi}$ (each multiplied by $z^{\|\upi\|}$ for the appropriate $\upi$) and is therefore rational, finishing the proof.
\end{proof}

We can now prove the main result of this section.

\begin{proof}[of Theorem \ref{thm:soln_sets_rational}.]
	Let $\cS$ denote an algebraic set. By Theorem \ref{thm:rationaltuples}, it suffices to show that $\cS$ is coset-wise polyhedral. By Lemma \ref{lem:virtual_sols} we have
	\begin{align*}
	\cS&= \bigcup_{(t_1, \ldots, t_n)\in B}\left\{(h_1t_1, \ \ldots, \ h_nt_n)\mid (h_1, \ \ldots, \ h_n)\in \cS_{(t_1,\ldots,t_n)}\right\} \\
	&= \bigcup_{(t_1,\ldots,t_n)\in T^n}\left\{(h_1t_1, \ \ldots, \ h_nt_n)\mid (h_1, \ \ldots, \ h_n)\in \cS_{(t_1,\ldots,t_n)} \right\}
	\end{align*}
	where each $\cS_{(t_1,\ldots,t_n)}$ is the solution set to some system of twisted equations in $\Z^k$ (and is empty for $(t_1, \ \ldots, \ t_n)\notin B$. By Lemma \ref{lem:twistedpoly}, each $\cS_{(t_1,\ldots,t_n)}$ is a polyhedral subset of $\Z^{kn}$, and thus $\cS$ is coset-wise polyhedral as required.
\end{proof}

For clarity, we explicitly state the description of algebraic sets in terms of polyhedral sets, which is a consequence of the proof of Theorem \ref{thm:soln_sets_rational}.
\begin{cor}\label{rem:algstatement}
	Let G be a finitely generated virtually abelian group (with a finite-index free abelian normal subgroup $\Z^k$ for some $k$). Choose a transversal $T$. Suppose $\cS\subset G^n$ is an algebraic set. Then for each $\mathbf{t}=(t_1, \ \ldots, \ t_n)\in T^n$, there exists a polyhedral set $\cS_{\mathbf{t}}\subseteq\Z^{kn}$ such that $\cS$ decomposes as a finite disjoint union:
	\[\cS = \bigcup_{\mathbf{t}\in T^n} \left\{(g_1,\ \ldots, \ g_n)\in \Z^kt_1\times\cdots\times\Z^kt_n\,\middle\vert\, (g_1t_1^{-1},\ \ldots, \ g_nt_n^{-1})\in\cS_{\mathbf{t}} \right\}. \]
\end{cor}

\subsection{Multivariate Growth Series}
We now turn to the \emph{multivariate} growth series (see Definition \ref{def:alggrowth}) and demonstrate that for an algebraic set $V$, the multivariate growth series $\mathbb{M}_{V\subseteq G^n,S}(z)$ is a \emph{holonomic} function.

\begin{dfn}\label{def:MGS}
	For clarity, we also define the multivariate growth series of a language. Let $L$ be a language over some finite weighted alphabet $A=\{a_1,\ldots,a_r\}$ (with weights denoted $\|a_i\|$) and let $|w|_i$ denote the number of occurrences of $a_i$ in a word $w\in L$. The \emph{weighted multivariate growth series} of $L$ is the formal power series
		\[\sum_{w\in L} z_1^{\|a_1\|\cdot|w|_1} z_2^{\|a_2\|\cdot|w|_2} \cdots z_r^{\|a_r\|\cdot|w|_r}\in\Q[[z_1,z_2,\ldots,z_r]].\]
\end{dfn}

Let $\mathbf{z}=(z_1,\ldots,z_n)$ and $\partial_{z_i}$ denote the partial derivative with respect to $z_i$.
\begin{dfn}
	A multivariate function $f(\mathbf{z})$ is \emph{holonomic} if the span of the set of partial derivatives
	\[\{\partial_{z_1}^{j_1}\partial_{z_2}^{j_2}\cdots\partial_{z_n}^{j_n} f(\mathbf{z})\mid j_i\in\Z_{\geq0}\}\] over the ring of rational functions $\C(\mathbf{z})$ is finite-dimensional.
\end{dfn}
From this definition, we see that a function is holonomic if and only if it satisfies a linear differential equation involving partial derivatives of finite order, and rational coefficients, for each variable $z_i$. Holonomic functions thus generalise the class of algebraic functions. For a more complete introduction to this topic, see \cite{Flajolet}.

In recent work \cite{Bishop}, Bishop extends results of Massazza \cite{Massazza} to show that a certain class of formal languages has holonomic multivariate growth series. The following Lemma follows easily from Proposition 4.3 of \cite{Bishop}, and the fact that holonomic functions are closed under algebraic substitution (Theorem B.3 of \cite{Flajolet}).
\begin{lem}\label{lem:polyholo}
	The weighted multivariate growth series of a polyhedral set (viewed as a formal language over the alphabet consisting of standard basis vectors) is holonomic.
\end{lem}

As in the univariate case, we prove a more general statement about coset-wise polyhedral subsets.
\begin{theorem}\label{thm:CWPholonomic}
	Let $V\subset G^n$ be a coset-wise polyhedral set of tuples of elements of a virtually abelian group $G$. Then the weighted multivariate growth series $\mathbb{M}_{V\subseteq G^n,S}$ is holonomic, with respect to any generating set $S$.
\end{theorem}
\begin{proof}
	Following the proof of Theorem \ref{thm:rationaltuples}, the coset-wise polyhedral set $V$ is represented by a finite disjoint union of polyhedral sets in $\Z^{kn}$, where $k$ is the dimension of the finite-index free abelian normal subgroup of $G$.

	Lemma \ref{lem:polyholo} implies that the weighted multivariate growth series of each of these polyhedral sets (in the sense of Definition \ref{def:MGS}) is holonomic. These series will involve $kn$ variables, say \[(z_{11},\ldots,z_{1k},\ z_{21},\ldots,z_{2k},\ \ldots,\ z_{n1},\ldots,z_{nk}).\] To obtain the weighted multivariate growth series of $V$ (in the sense of Definition \ref{def:alggrowth}), we need only set each $z_{ij}=z_i$ and multiply each of the finitely many growth series by an appropriate constant to account for the contribution from each pattern $\upi$. The closure properties of holonomic functions (Theorem B.3 of \cite{Flajolet}) ensure that the resulting growth series is still holonomic (with variables $z_1,\ldots,z_n$ corresponding to the variables in the system of equations).
\end{proof}
\begin{cor}\label{cor:holonomic}
	An algebraic set in a virtually abelian group has holonomic weighted multivariate growth series.
\end{cor}
\begin{proof}
	The proof of Theorem \ref{thm:soln_sets_rational} shows that any algebraic set is coset-wise polyhedral.
\end{proof}

\section{Further work}
	It is hoped that Corollary \ref{cor:holonomic} can be improved upon. Many power series associated to structures in virtually abelian groups turn out to be rational (see \cite{Benson}, \cite{Evetts}, \cite{Bishop}) and therefore we make the following conjecture, noting that an affirmative answer would immediately imply Theorem \ref{thm:soln_sets_rational}.
	\begin{conjecture}
		The weighted multivariate growth series of any algebraic set of a finitely generated virtually abelian group is rational.
	\end{conjecture}

	A system of equations in a group $G$ is an example of a first order sentence, part of the first order theory of $G$. For groups with decidable first order theory (including virtually abelian groups \cite{elementary_gp_theories}), a natural generalisation is to study the sets of tuples of elements of $G$ that satisfy more general first order sentences. The present paper is intended to be the first step in an investigation of the formal language properties, and the growth series behaviour, of such more general \emph{definable sets}.

\section*{Acknowledgements}
We wish to thank Laura Ciobanu for her invaluable advice and guidance, as well as the anonymous referee for detailed comments and suggestions. Evetts was partially supported by EPSRC grant EP/R035814/1, and also wishes to thank the London Mathematical Society and the Erwin Schr\"odinger International Institute for Mathematics and Physics for support during the writing of this paper.

%\nocite{intro_num_theory}
\nocite{appl_L_systems_GT}
%\nocite{o_decid_conj_prob}
\nocite{math_theory_L_systems}
\nocite{handbook_form_lang}
\nocite{eqns_hyp_grps_conf}
\bibliography{references}
\bibliographystyle{abbrv}
\end{document}